\documentclass[numbers,webpdf,imaiai]{ima-authoring-template}

\graphicspath{{Fig/}}

\theoremstyle{thmstyletwo}%
\newtheorem{theorem}{Theorem}
\newtheorem{proposition}[theorem]{Proposition}%

\newtheorem{example}{Example}%
\newtheorem{remark}{Remark}%

\newtheorem{definition}{Definition}

\newtheorem{lemma}[theorem]{Lemma}
\newtheorem{corollary}[theorem]{Corollary}

\numberwithin{equation}{section}

\usepackage{mathtools}
\mathtoolsset{centercolon}  

\numberwithin{equation}{section} 

\definecolor{mygreen}{rgb}{0.1,0.75,0.2}

\providecommand{\mathbbm}{\mathbb} 

\newcommand{\R}{\mathbbm{R}}

\newcommand{\F}{\mathcal{F}}

\renewcommand{\phi}{\varphi}

\usepackage[font = small, margin=30pt]{caption}

\usepackage[]{algorithm}
\usepackage{algpseudocode}
\usepackage{enumerate}

\usepackage{graphicx}

 \counterwithin{theorem}{section}
  \counterwithin{definition}{section}
  \counterwithin{remark}{section}

\usepackage{comment}

\newcommand{\iid}{\stackrel{\text{i.i.d.}}{\sim}}

\newcommand{\E}{\mathbb{E}}

\renewcommand{\S}{\mathbb{S}}

\renewcommand{\P}{\mathbb{P}}

\newcommand{\mcF}{\mathcal{F}}

\newcommand{\mcH}{\mathcal{H}}

\usepackage{multirow}

\newcommand{\inparen}[1]{\left(#1\right)}             
\newcommand{\insquare}[1]{\left[#1\right]}             

\usepackage[scr = esstix, cal = cm, frak=euler]{mathalfa}

\definecolor{mygreen}{rgb}{0.1,0.75,0.2}

\begin{document}

\DOI{DOI HERE}
\copyrightyear{2021}
\vol{00}
\pubyear{2021}
\access{Advance Access Publication Date: Day Month Year}
\appnotes{Paper}
\copyrightstatement{Published by Oxford University Press on behalf of the Institute of Mathematics and its Applications. All rights reserved.}
\firstpage{1}



\title[Sharp concentration of simple random tensors II: asymmetry]{Sharp concentration of simple random tensors II: asymmetry}


\author{Jiaheng Chen*
\address{\orgdiv{Committee on Computational and Applied Mathematics}, \orgname{University of Chicago}, \orgaddress{ \postcode{IL 60637}, \country{USA}}}}
\author{Daniel Sanz-Alonso
\address{\orgdiv{Department of Statistics}, \orgname{University of Chicago}, \orgaddress{ \postcode{IL 60637}, \country{USA}}}}


\authormark{J. Chen and D. Sanz-Alonso}

\corresp[*]{Corresponding author: \href{email:jiaheng@uchicago.edu}{jiaheng@uchicago.edu}}

\received{Date}{0}{Year}
\revised{Date}{0}{Year}
\accepted{Date}{0}{Year}


\abstract{
This paper establishes sharp concentration inequalities for simple random tensors. Our theory unveils a phenomenon that arises only for asymmetric tensors of order $p \ge 3:$ when the effective ranks of the covariances of the component random variables lie on both sides of a critical threshold, an additional logarithmic factor emerges that is not present in sharp bounds for symmetric tensors. 
 To establish our results, we develop empirical process theory for products of $p$ different function classes evaluated at $p$ different random variables, extending generic chaining techniques for quadratic and product empirical processes to higher-order settings. 
}



\maketitle

\section{Introduction}\label{sec:introduction}

This paper establishes sharp bounds for the operator-norm deviation of the sum of simple (rank-one) random tensors from its expectation. For integer $p\ge 2$ and each $1\le k\le p$, let $X^{(k)},X^{(k)}_1,\ldots,X^{(k)}_N$ be i.i.d. centered sub-Gaussian random variables in a separable Hilbert space $H^{(k)}$ with covariance operator $\Sigma^{(k)}$. Our first main result, Theorem \ref{thm:main1}, shows that, with a standard notation described below,
\[
\E \bigg\|\frac{1}{N} \sum_{i=1}^N X_i^{(1)} \otimes \cdots \otimes X_i^{(p)}-\E\, X^{(1)} \otimes \cdots \otimes X^{(p)}\bigg\| \lesssim_p \bigg(\prod_{k=1}^p\|\Sigma^{(k)}\|^{1 / 2}\bigg)\mathscr{E}_N\big((\Sigma^{(k)})_{k=1}^p\big),
\]
where
\[
\mathscr{E}_N\big((\Sigma^{(k)})_{k=1}^p\big):= \bigg(\frac{\sum_{k=1}^p r(\Sigma^{(k)})}{N}\bigg)^{1/2}+\frac{1}{N}\prod_{k=1}^{p}\Big(r(\Sigma^{(k)})+\log N\Big)^{1/2},\quad \ r(\Sigma^{(k)}):=\frac{\mathrm{Tr}(\Sigma^{(k)})}{\|\Sigma^{(k)}\|}.
\]
We also obtain corresponding $q$-th moment bounds for $q\ge 1$. In addition, Theorem \ref{thm:main1} establishes a matching lower bound in the case of independent Gaussian data, demonstrating the sharpness of the result and showing that the logarithmic factors are, in general, unavoidable. When $p=2$, our result yields a sharp dimension-free concentration bound for the sample cross-covariance operator of sub-Gaussian random variables $X^{(1)} $ and $X^{(2)}$, without imposing any assumptions on their correlation structure:
\begin{align*}
&\E \bigg\|\frac{1}{N}\sum_{i=1}^{N}X^{(1)}_i\otimes X^{(2)}_i-\E\, X^{(1)}\otimes X^{(2)}\bigg\| \\
& \lesssim \big(\|\Sigma^{(1)}\|\|\Sigma^{(2)}\|\big)^{1/2}\left(\bigg(\frac{r(\Sigma^{(1)})+r(\Sigma^{(2)})}{N}\bigg)^{1/2}+\frac{(r(\Sigma^{(1)})r(\Sigma^{(2)}))^{1/2}}{N}\right).
\end{align*} To prove Theorem \ref{thm:main1}, we introduce and develop the theory of multi-product empirical processes. In our second main result, Theorem \ref{thm:main2}, we leverage generic chaining techniques to obtain a sharp in-expectation bound on
 \begin{equation*}
 \sup_{f^{(k)} \in \mathcal{F}^{(k)},1\le k\le p}\bigg|\frac{1}{N} \sum_{i=1}^N \prod_{k=1}^p f^{(k)} (X^{(k)}_i)-\mathbb{E} \prod_{k=1}^p f^{(k)} (X^{(k)})\bigg|
\end{equation*}
in terms of quantities that reflect the geometric complexity of the function classes $\F^{(k)}, 1\le k\le p,$ which are taken to contain bounded linear functionals on the corresponding Hilbert spaces $H^{(k)}$ in the proof of Theorem \ref{thm:main1}.

\subsection{ Main 
 contributions and outline}
This paper further develops the study of sharp concentration of simple random tensors initiated in \cite{al2025sharp}. While the analysis in \cite{al2025sharp} is limited to \emph{symmetric} tensors,
this paper generalizes the theory to \emph{asymmetric} tensors. In so doing, 
we uncover a new phenomenon that only arises in the concentration of asymmetric tensors of order $p \ge 3$. To establish our results, we investigate a more general type of empirical process than \cite{al2025sharp}, for products of $p$ different function classes evaluated at $p$ different random variables. A key technical contribution is Theorem \ref{thm:key_lemma}, which establishes a sharp uniform bound on the $\ell_2$-norm of the product of the coordinate projection vectors associated with the different function classes. 

The proof of Theorem \ref{thm:key_lemma} is based on generic chaining across multiple function classes and builds on ideas of Mendelson \cite{mendelson2016upper} for controlling multiplier and product processes. The main new difficulty is that we must handle heterogeneous classes $\F^{(1)},\ldots,\F^{(p)}$ evaluated at different random variables, and exploit the resulting multilinear structure to carry out a coordinated chaining argument across the classes.  A crucial step in controlling the increments is to incorporate $\ell_{\infty}$-norm bounds on the relevant coordinate projection vectors at appropriate stages of the chaining argument; this introduces $\mathsf{polylog}(N)$ factors. While these terms may appear crude at first glance, they are in fact sharp and unavoidable in general. In particular, even if some class $\F^{(k)}$ is ``simple'' in a global complexity sense, its contribution to the overall bound must still include at least a term that reflects its largest coordinate fluctuation in $\ell_{\infty}$. This viewpoint provides a principled way to understand the behavior of products of coordinate projection vectors across heterogeneous classes $(\F^{(k)})_{k=1}^{p}$ and different underlying random variables ---an effect that is not captured by the approach in \cite{al2025sharp}. 
Remarks \ref{rem:main1_1} and \ref{rem:general-multi-product-process} further discuss how the results and techniques in this paper differ from those in \cite{al2025sharp}.

In the special case $p=2$, our results yield a new dimension-free concentration bound for the sample cross-covariance operator; see Remark \ref{rem:main1_1}. Cross-covariances are a core building block in many procedures, including principal component methods, canonical correlation analysis, and more broadly in the estimation of second-order structure in multivariate or functional data. Obtaining dimension-independent deviation bounds is particularly valuable in high- or infinite-dimensional regimes, where classical bounds scale poorly with the ambient dimension and become uninformative. Our estimates provide a useful theoretical tool for such settings. More generally, we anticipate that Theorem \ref{thm:main1} may have applications in various tensor-related statistical problems \cite{mccullagh2018tensor,bi2021tensors,auddy2024tensors}, including tensor estimation \cite{han2022optimal,diakonikolas2024implicit}, tensor completion \cite{yuan2016tensor,xia2021statistically}, tensor regression \cite{luo2024tensor}, and tensor singular value decomposition \cite{zhang2018tensor,zhang2019optimal}.  Our results may also be useful for independent component analysis, where estimation of fourth-order moment tensors (kurtosis) plays a central role \cite{auddy2023large}. Beyond tensor methodology, our findings may inform work in ensemble Kalman methods for inverse problems and data assimilation \cite{sanzstuarttaeb,bach2024machine,chada2021iterative}, learning theory and empirical risk minimization \cite{even2021concentration}, and the method of moments \cite{sherman2020estimating}, with potential applications to cryo-EM and multi-reference alignment problems \cite{perry2019sample,bandeira2020optimal,dou2024rates}.

Section \ref{sec:main} states  our two main results and discusses related work. We study concentration of simple random tensors and multi-product empirical processes in Sections \ref{sec:concentration} and \ref{sec:multi}, respectively.

\subsection{Notation}
Given two positive sequences $\{a_n\}$ and $\{b_n\}$, we write $a_n\lesssim b_n$ to denote that $a_n \le c b_n$ for some absolute constant $c>0$. If both $a_n \lesssim b_n$ and $b_n \lesssim a_n$ hold simultaneously, we write $a_n \asymp b_n$. If the constant $c$ depends on some parameter $\tau$, we use the notations $a_n \lesssim_{\tau} b_n,b_n\lesssim_{\tau} a_n$, and $a_n\asymp_{\tau} b_n$ to indicate this dependence. We denote the unit Euclidean sphere in $\R^d$ by $\S^{d-1}:=\{u\in \R^d: \|u\|_{\ell_2}=1\}$. For a vector $x=(x_i)_{i=1}^{N}\in \R^{N}$, we denote by $\left(x_i^*\right)_{i=1}^N$  the non-increasing rearrangement of $\left(\left|x_i\right|\right)_{i=1}^N$. 

\section{Main results}\label{sec:main}
Here we state the two main theorems of the paper and compare them with existing results. 
Subsection \ref{subsec:cencentraion_of_tensor} focuses on the concentration of simple random tensors and
Subsection \ref{subsec:multi-product} focuses on multi-product empirical processes.

\subsection{Concentration of simple random tensors}\label{subsec:cencentraion_of_tensor}

For integer $p\ge 2$ and $1\le k\le p$, let $(H^{(k)}, \langle \cdot, \cdot\rangle_{H^{(k)}}, \| \cdot \|_{H^{(k)}})$ be a separable real Hilbert space, and let $X^{(k)}$ be a centered $H^{(k)}$-valued random variable. Suppose that $\E|\langle X^{(k)}, v \rangle_{H^{(k)}}|^p < \infty$ for all $v \in U_{H^{(k)}}:= \{v \in H^{(k)} : \|v\|_{H^{(k)}} = 1 \}$ and let  $X^{(1)} \otimes \cdots \otimes X^{(p)}$ be the multilinear form
\begin{equation*}
(X^{(1)} \otimes \cdots \otimes X^{(p)}) (v_1, \ldots, v_p) := \langle X^{(1)},v_1 \rangle_{H^{(1)}} \cdots \langle X^{(p)}, v_p \rangle_{H^{(p)}}, \quad \forall \ (v_1, \ldots, v_p) \in H^{(1)} \times \cdots \times H^{(p)}.
\end{equation*}
Thus, $X^{(1)} \otimes \cdots \otimes X^{(p)}$ is a random variable taking values in the tensor space $H^{(1)}\otimes \cdots \otimes H^{(p)}$. We refer for instance to \cite[Section II.4]{reed1980methods} for background on tensor products of Hilbert spaces. When $p=2$ and $X^{(1)}=X^{(2)}=X\in H$, the second-order moment tensor $\E\, X^{\otimes 2}(v_1,v_2) = \E \langle X, v_1\rangle_H \langle X,v_2\rangle_H,$  $\, \forall\, (v_1, v_2) \in H \times H $ is called the \emph{covariance} of $X.$ By the Riesz representation theorem, this bilinear form can be represented by the \emph{covariance operator} $\Sigma: H \to H$ defined by 
$$ \langle v_1, \Sigma v_2 \rangle_H : =  \E \,X^{\otimes 2}(v_1,v_2), \qquad    \forall \ (v_1,v_2) \in H \times H.$$

 Given independent copies $X^{(k)}_1, \ldots, X^{(k)}_N$ of $X^{(k)}$ for $1\le k\le p$, our goal is to establish concentration inequalities for the sum of simple (rank-one) random tensors $\frac{1}{N} \sum_{i=1}^N X^{(1)}_i\otimes \cdots \otimes X^{(p)}_i$.   Specifically, we seek to obtain upper bounds on the deviation 
\begin{align}\label{eq:opnormdeviation}
&\bigg\|\frac{1}{N} \sum_{i=1}^N X_i^{(1)} \otimes \cdots \otimes X_i^{(p)}-\E\, X^{(1)} \otimes \cdots \otimes X^{(p)}\bigg\| \nonumber\\
&\quad = \sup_{v_k \in U_{H^{(k)}},1\le k\le p} \bigg| \frac{1}{N}\sum_{i=1}^N \langle X^{(1)}_i,v_1\rangle_{H^{(1)}} \cdots \langle X^{(p)}_i, v_p \rangle_{H^{(p)}}  -  \E\langle X^{(1)},v_1 \rangle_{H^{(1)}} \cdots \langle X^{(p)}, v_p \rangle_{H^{(p)}} \bigg|,
 \end{align}
where $\| \cdot \|$ denotes the norm in the tensor space $H^{(1)}\otimes \cdots \otimes H^{(p)}.$ Notice that for $p=2$ and $X^{(1)}=X^{(2)}=X\in H$, the operator norm of the covariance operator $\Sigma: H \to H$ agrees with the norm of the covariance  $\E \,X^{\otimes 2}$ as a bilinear form in $H^{\otimes 2}.$ Consequently, we will slightly abuse notation and  denote by $\| \Sigma\|$ the operator norm of the covariance operator. Furthermore, to simplify the notation, we will henceforth omit the subscript $H^{(k)}$ in the inner-product $\langle \cdot, \cdot \rangle_{H^{(k)}}$ and norm $\| \cdot \|_{H^{(k)}}$. 

We will focus on Gaussian and sub-Gaussian random variables. Recall that an $H$-valued random variable $X$ is called \emph{(sub)-Gaussian} iff all one-dimensional projections $\langle X, v \rangle$ for $v \in H$ are (sub)-Gaussian real-valued random variables. Furthermore, $X$ is called \emph{pre-Gaussian} iff there exists a centered Gaussian random variable $Y$ in $H$ with the same covariance operator as that of $X.$

For a covariance operator $\Sigma$, the \emph{effective rank} is defined as \[
r(\Sigma): = \frac{\mathrm{Tr}(\Sigma)}{\| \Sigma\|}.
\]
Our first main result, Theorem \ref{thm:main1}, provides a sharp dimension-free bound for the deviation \eqref{eq:opnormdeviation} in terms of the effective ranks  $r(\Sigma^{(1)}),\ldots, r(\Sigma^{(p)})$ of the covariance operators associated with the random variables $X^{(1)},\ldots, X^{(p)}$. We prove Theorem \ref{thm:main1} in Section \ref{sec:concentration}.

\begin{theorem}\label{thm:main1}
     For any integer $p\ge 2$ and $1 \leq k \leq p$, let $X^{(k)}, X_1^{(k)}, \ldots, X_N^{(k)}$ be i.i.d. centered sub-Gaussian and pre-Gaussian random variables in $H^{(k)}$ with covariance operator $\Sigma^{(k)}$. Then, for any $q\ge 1$,
\[
\left(\E \bigg\|\frac{1}{N} \sum_{i=1}^N X_i^{(1)} \otimes \cdots \otimes X_i^{(p)}-\E\, X^{(1)} \otimes \cdots \otimes X^{(p)}\bigg\|^q \right)^{1/q} \lesssim_{p,q} \bigg(\prod_{k=1}^p\|\Sigma^{(k)}\|^{1 / 2}\bigg)\mathscr{E}_N\big((\Sigma^{(k)})_{k=1}^p\big),
\]
where
\begin{align*}
\mathscr{E}_N\big((\Sigma^{(k)})_{k=1}^p\big):= \bigg(\frac{\sum_{k=1}^p r(\Sigma^{(k)})}{N}\bigg)^{1/2}+\frac{1}{N}\prod_{k=1}^{p}\Big(r(\Sigma^{(k)})+\log N\Big)^{1/2}.
\end{align*}
Moreover, if $X^{(1)},\ldots, X^{(p)}, (X^{(1)}_i)_{i=1}^{N},\ldots,(X^{(p)}_i)_{i=1}^N$ are independent Gaussian, then
\[
\left(\E \bigg\|\frac{1}{N} \sum_{i=1}^N X_i^{(1)} \otimes \cdots \otimes X_i^{(p)}-\E\, X^{(1)} \otimes \cdots \otimes X^{(p)}\bigg\|^q \right)^{1/q}\asymp_{p,q}\bigg(\prod_{k=1}^p\|\Sigma^{(k)}\|^{1 / 2}\bigg)\mathscr{E}_N\big((\Sigma^{(k)})_{k=1}^p\big).
\]
\end{theorem}

To the best of our knowledge, Theorem \ref{thm:main1} is the first result in the literature to establish \emph{sharp dimension-free} bounds for the deviation \eqref{eq:opnormdeviation} under sub-Gaussian assumptions. Notably, the upper bound holds without requiring independence between the sequences $(X^{(k)}_i)_{i=1}^{N}$ and $(X^{(k')}_i)_{i=1}^{N}$ for $k\ne k'$; that is, no assumptions are made on the correlation structure across components. Moreover, Theorem \ref{thm:main1} shows that our expectation bound is sharp in the case of independent Gaussian data.

The remainder of this subsection provides a detailed discussion of Theorem \ref{thm:main1} and its connections to existing results. In Remark \ref{rem:main1_1}, we compare our result to existing concentration inequalities for sums of simple symmetric tensors and derive a straightforward corollary for finite-dimensional random tensors. Remark~\ref{rem:log_concave} presents an additional consequence of our theory concerning isotropic, unconditional, log-concave ensembles.

\begin{remark}\label{rem:main1_1}

For simplicity, in this remark we focus on the first-moment case $q=1$. When $p=2$, Theorem \ref{thm:main1} yields the following dimension-free concentration bound for the sample cross-covariance operator:
\begin{align*}
&\E \bigg\|\frac{1}{N}\sum_{i=1}^{N}X^{(1)}_i\otimes X^{(2)}_i-\E\, X^{(1)}\otimes X^{(2)}\bigg\|\\
&\lesssim \big(\|\Sigma^{(1)}\|\|\Sigma^{(2)}\|\big)^{1/2}\left(\bigg(\frac{r(\Sigma^{(1)})+r(\Sigma^{(2)})}{N}\bigg)^{1/2}+\frac{(r(\Sigma^{(1)})+\log N)^{1/2}(r(\Sigma^{(2)})+\log N)^{1/2}}{N}\right)\\
&\asymp \big(\|\Sigma^{(1)}\|\|\Sigma^{(2)}\|\big)^{1/2}\left(\bigg(\frac{r(\Sigma^{(1)})+r(\Sigma^{(2)})}{N}\bigg)^{1/2}+\frac{(r(\Sigma^{(1)})r(\Sigma^{(2)}))^{1/2}}{N}\right),
\end{align*}
which improves upon the earlier bound in \cite[Lemma A.3]{ghattas2022non}. To our knowledge, this dimension-free concentration bound for the sample cross-covariance operator is new. It holds without any assumption on the dependence between the sequences $(X_i^{(1)})_{i=1}^{N}$ and $(X_i^{(2)})_{i=1}^{N}$, provided that both are sub-Gaussian and pre-Gaussian random variables. Moreover, it is sharp (up to constants) in view of the matching lower bound in the independent case established in Theorem~\ref{thm:main1}. 

When $X^{(1)}=\cdots = X^{(p)}=X$ and $X_i^{(1)}=\cdots = X_i^{(p)}=X_i \,$ for all $1\le i\le N,$ so that $\Sigma^{(1)}=\cdots=\Sigma^{(p)}=\Sigma,$ Theorem \ref{thm:main1} yields
\begin{align}\label{eq:rem_aux1}
\E \bigg\|\frac{1}{N}\sum_{i=1}^{N} X_i^{\otimes p}-\E\, X^{\otimes p}\bigg\| \lesssim_p \|\Sigma\|^{p/2} \bigg(\sqrt{\frac{r(\Sigma)}{N}}+\frac{(r(\Sigma)+\log N)^{p/2}}{N}\bigg) \asymp_p \|\Sigma\|^{p/2} \bigg(\sqrt{\frac{r(\Sigma)}{N}}+\frac{r(\Sigma)^{p/2}}{N}\bigg),
\end{align}
where the final equivalence follows from the bound $(\log N)^{p/2}/N\lesssim_p 1/\sqrt{N}$. This recovers the sharp concentration inequality for simple symmetric tensors established in \cite[Theorem 2.1]{al2025sharp}. We refer the reader to \cite{al2025sharp} for further references on the study of \eqref{eq:rem_aux1}, as well as related literature on sample covariance operators and random tensors.

Since the operator-norm deviation under consideration is invariant under permutations of the component random vectors, we may assume without loss of generality that the effective ranks satisfy $r(\Sigma^{(1)})\ge r(\Sigma^{(2)})\ge \cdots\ge r(\Sigma^{(p)})$. Suppose there exists an index $t$ such that $r(\Sigma^{(t)}) \ge \log N \ge r(\Sigma^{(t+1)})$. Then, the upper bound in Theorem \ref{thm:main1} simplifies to
\begin{align*}
&\mathbb{E} \bigg\| \frac{1}{N} \sum_{i=1}^N X_i^{(1)} \otimes \cdots \otimes X_i^{(p)}
- \mathbb{E}\, X^{(1)} \otimes \cdots \otimes X^{(p)} \bigg\| \\
&\quad \lesssim_p 
\bigg( \prod_{k=1}^p \|\Sigma^{(k)}\|^{1/2} \bigg)
\left( \bigg( \frac{r(\Sigma^{(1)})}{N} \bigg)^{1/2}
+ \frac{(\log N)^{(p-t)/2}}{N} \prod_{k=1}^{t} r(\Sigma^{(k)})^{1/2} \right).
\end{align*}
As will be shown in Proposition \ref{prop:lowerbound},  the upper bound can be reversed (up to a constant depending only on $p$) in the case of independent Gaussian data,  indicating that the logarithmic factors in the upper bound are, in general, unavoidable. This new phenomenon arises only for asymmetric tensors of order $p\ge 3$ ---specifically, when the effective ranks of the component covariances vary across the threshold $\log N$, with some exceeding and others falling below it. The random variables in the tensor product whose covariances have effective ranks smaller than $\log N$ contribute logarithmic factors to the second term in the upper bound. 
Since this behavior only arises for asymmetric tensors of order $p \ge 3,$
it has not been observed in previous work on the concentration of simple symmetric tensors or sample covariance operators.

A corresponding finite-dimensional corollary is as follows. For each $1\le k\le p$, let $X^{(k)}, X_1^{(k)}, \ldots, X_N^{(k)} $ be i.i.d. standard Gaussian random vectors in $\R^{d_k}$, so that $\Sigma^{(k)}=I_{d_k}$ and $r(\Sigma^{(k)})=\mathrm{Tr}(\Sigma^{(k)})/\|\Sigma^{(k)}\|=d_k$. If 
\[d_1\ge \cdots \ge d_t \ge \log N \ge d_{t+1} \ge \cdots \ge d_p,
\] 
then
\begin{align*}
\E \bigg\|\frac{1}{N} \sum_{i=1}^N X_i^{(1)} \otimes \cdots \otimes X_i^{(p)}-\E\, X^{(1)} \otimes \cdots \otimes X^{(p)}\bigg\| \lesssim_p  \bigg(\frac{d_1}{N}\bigg)^{1/2}+\frac{(\log N)^{(p-t)/2}}{N}\prod_{k=1}^{t}d_k^{1/2}.
\end{align*}
\end{remark}

\begin{remark}\label{rem:log_concave}

A minor modification of our proof (see Subsection \ref{subsec:log-concave}) yields the following result for isotropic, unconditional, log-concave ensembles. 

Recall that a probability measure $\nu$ on $\R^d$ is log-concave if, for any $t\in [0,1]$ and any nonempty Borel measurable sets $A, B\subset \R^d$, the following inequality holds:
\[
\nu(tA+(1-t)B)\ge \nu(A)^t \nu(B)^{1-t}.
\]
For each $1 \leq k \leq p$, let $X^{(k)}$ be a random vector in $\R^{d_k}$, distributed according to an isotropic, unconditional,
log-concave measure. That is, for all $v\in \R^{d_k}$, we have $\E \langle X,v \rangle^2=\|v\|_{\ell_2}^2$; the distribution of $X^{(k)}$ is invariant under coordinate-wise sign changes; and the law of $X^{(k)}$ has a log-concave density. Let $X_1^{(k)}, \ldots, X_N^{(k)}$ be independent copies of $X^{(k)}$, for each $1\le k\le p$. Then, the following bound holds:
\[
\E \bigg\|\frac{1}{N} \sum_{i=1}^N X_i^{(1)} \otimes \cdots \otimes X_i^{(p)}-\E\, X^{(1)} \otimes \cdots \otimes X^{(p)}\bigg\| \lesssim_p \bigg(\frac{\sum_{k=1}^p d_k}{N}\bigg)^{1/2}+\frac{1}{N}\prod_{k=1}^{p}\big(d_k^{1/2}+\log N\big).
\]
In particular, if
\[
d_1\ge \cdots \ge d_t \ge (\log N)^2 \ge d_{t+1} \ge \cdots \ge d_p,
\] 
then
\begin{align*}
\E \bigg\|\frac{1}{N} \sum_{i=1}^N X_i^{(1)} \otimes \cdots \otimes X_i^{(p)}-\E\, X^{(1)} \otimes \cdots \otimes X^{(p)}\bigg\| \lesssim_p  \bigg(\frac{d_1}{N}\bigg)^{1/2}+\frac{(\log N)^{p-t}}{N}\prod_{k=1}^{t}d_k^{1/2}.
\end{align*}
We further remark that the corresponding $q$-th moment bound can be derived by integrating the tail estimate from Subsection~\ref{subsec:log-concave}. We omit the details for brevity and focus on the first-moment bound. 
\end{remark}

\subsection{Multi-product empirical processes}\label{subsec:multi-product}
This subsection contains the statement of our second main result, Theorem~\ref{thm:main2}, which gives a general upper bound on the suprema of multi-product empirical processes. Theorem~\ref{thm:main2} will be key to establish our upper bounds in Theorem \ref{thm:main1}. For $1\le k\le p$, let $X^{(k)}, X^{(k)}_1, \ldots, X^{(k)}_N \iid \mu^{(k)}$ be a sequence of random variables on a probability space $(\Omega^{(k)}, \mu^{(k)})$. We consider the \emph{multi-product} centered empirical process indexed by $p$ (potentially distinct) function classes $\F^{(1)},\ldots,\F^{(p)}$, where each $\F^{(k)}$ is defined on $(\Omega^{(k)}, \mu^{(k)})$. The process is given by 
\begin{align}\label{eq:multi-productempiricalprocess}
 (f^{(1)},\ldots,f^{(p)})\mapsto \frac{1}{N}\sum_{i=1}^N \prod_{k=1}^p f^{(k)}(X^{(k)}_i)-\E \prod_{k=1}^p f^{(k)}(X^{(k)}),\qquad f^{(k)}\in \F^{(k)}, \ 1\le k\le p.
\end{align}
In our proof of Theorem \ref{thm:main1}, the class $\F^{(k)}$ will be taken to contain all linear functionals in $H^{(k)}$ with norm one. 

For $p = 1$, \eqref{eq:multi-productempiricalprocess} recovers the standard empirical process \cite{van2023weak, talagrand2022upper}, and for $p = 2$ the quadratic and product empirical processes extensively studied, e.g., in \cite{rudelson1999random,klartag2005empirical,mendelson2007reconstruction,adamczak2010quantitative,mendelson2010empirical,mendelson2012generic,dirksen2015tail,bednorz2014,mendelson2016upper,bednorz2016bounds, koltchinskii2017concentration,han2022exact}. Beyond their theoretical significance, quadratic and product empirical processes play a crucial role in various statistical applications, such as covariance estimation \cite{koltchinskii2017concentration}. Our goal is to bound
\[
\sup_{f^{(k)} \in \mathcal{F}^{(k)},1\le k\le p}\bigg|\frac{1}{N} \sum_{i=1}^N \prod_{k=1}^p f^{(k)} (X^{(k)}_i)-\mathbb{E} \prod_{k=1}^p f^{(k)} (X^{(k)})\bigg|
\]
for any $p\ge 2$ in terms of quantities that reflect the geometric complexity of the indexing classes $(\F^{(k)})_{k=1}^{p}.$ To that end,  for any function $f$ on $(\Omega, \mu)$, we introduce the Orlicz $\psi_2$-norm of $f$ given by
\[
\|f\|_{\psi_2}:=\inf \left\{c>0: 
\mathbb{E}_{X \sim \mu}\insquare{ \exp \inparen{\frac{|f(X)|^2}{c^2}}} \leq 2\right\},
\]
with the convention $\inf \emptyset = \infty.$ 
Our upper bound will depend on 
\[
d_{\psi_2}(\F) := \sup_{f \in \F}\|f\|_{\psi_2}
\]
and on Talagrand's functional $\gamma(\mcF, \psi_2)$, whose definition we now recall.

\begin{definition}[{Talagrand's $\gamma$ functional}]\label{def:admissible_sequence_gamma_functional}
Let $(\F, d)$ be a metric space. An admissible sequence is an increasing sequence $(\F_s)_{s\ge 0} \subset \F$ which satisfies $\F_s \subset \F_{s+1}, \left|\F_0\right|=1, \left|\F_s\right| \leq 2^{2^s}$ for $s \ge 1,$ and $\bigcup_{s=0}^{\infty}\F_s$ is dense in $\F$. 
Let
\[
\gamma(\F, d) :=\inf \sup _{f \in \F} \sum_{s \geq 0} 2^{s /2} d\left(f, \F_s\right),
\]
where the infimum is taken over all admissible sequences. We write $\gamma(\mcF, \psi_2)$ when the distance on $\mathcal{\F}$ is induced by the $\psi_2$-norm.
\end{definition}

We are ready to state our second main result. Its proof can be found in Section \ref{sec:multi}.

\begin{theorem}\label{thm:main2} For any integer $p\ge 2$ and $1\le k\le p$, let $X^{(k)}, X^{(k)}_1, \ldots, X^{(k)}_N \iid \mu^{(k)}$ be a sequence of random variables on the probability space $(\Omega^{(k)}, \mu^{(k)})$, and let $\F^{(k)}$ be a class of functions defined on $(\Omega^{(k)},\mu^{(k)})$. Assume $0\in \F^{(k)}$ or $\F^{(k)}$ is symmetric (i.e., $f^{(k)}\in \F^{(k)}$ implies $-f^{(k)}\in \F^{(k)}$). Then, for any $q\ge 1$,
\begin{align*}
\left(\E \sup_{f^{(k)} \in \mathcal{F}^{(k)},1\le k\le p}\bigg|\frac{1}{N} \sum_{i=1}^N \prod_{k=1}^p f^{(k)} (X^{(k)}_i)-\mathbb{E} \prod_{k=1}^p f^{(k)} (X^{(k)})\bigg|^q\right)^{1/q} \lesssim_{p,q} \bigg(\prod_{k=1}^{p} d_{\psi_2}(\F^{(k)})\bigg) \mathscr{E}_N\big( (\F^{(k)})_{k=1}^p\big) ,
\end{align*}
where
\[
\mathscr{E}_N\big((\F^{(k)})_{k=1}^p\big)
:= \frac{\sum_{k=1}^{p}\bar{\gamma}(\F^{(k)},\psi_2)}{\sqrt{N}}+\frac{\prod_{k=1}^{p}\big(\bar{\gamma}(\F^{(k)},\psi_2)+(\log N)^{1/2}\big)}{N},\quad \bar{\gamma}(\F^{(k)},\psi_2):=\frac{\gamma(\F^{(k)},\psi_2)}{d_{\psi_2}(\F^{(k)})}.
\]
\end{theorem}

\begin{remark}\label{rem:general-multi-product-process} 
 
Here, for each $1\le k\le p$, the quantities $\gamma(\F^{(k)},\psi_2)$ and $d_{\psi_2}(\F^{(k)})$ are defined with respect to the probability measure $\mu^{(k)}$ on the probability space $(\Omega^{(k)},\mu^{(k)})$. Moreover, the sequences $\big(X_i^{(k)}\big)_{i=1}^N$ and $\big(X_i^{(k^{\prime})}\big)_{i=1}^{N}$ are not necessarily independent for $k\ne k^{\prime}$.
This result generalizes the upper bound established in \cite[Theorem 1.13]{mendelson2016upper} for the supremum of product empirical processes indexed by function classes $\mcF$ and $\mcH$, that is,
\[
(f,h) \mapsto \frac{1}{N} \sum_{i=1}^N f(X_i)h(X_i)-\E fh,\qquad f\in \F, \ h\in \mcH,
\]
extending it to higher-order settings, including the case where each factor in the product involves a different random variable.

In the special case where $\Omega^{(1)}=\cdots=\Omega^{(p)}=\Omega, \mu^{(1)}=\cdots=\mu^{(p)}=\mu, \F^{(1)}= \cdots=\F^{(p)}=\F$, and $X^{(1)}=\cdots =X^{(p)}=X$, with $X^{(1)}_i=\cdots = X^{(p)}_i=X_i \ $ for all $1\le i\le N,$ Theorem \ref{thm:main2} (with $q=1$) implies
\begin{align}\label{eq:rem_aux2}
\E \sup_{f\in \F} \bigg|\frac{1}{N}\sum_{i=1}^N f^p(X_i) -\E f^p(X)\bigg|&\le \E \sup_{f^{(k)}\in \F,1\le k\le p} \bigg|\frac{1}{N}\sum_{i=1}^N \prod_{k=1}^{p}f^{(k)}(X_i) -\E f^p(X)\bigg| \nonumber\\
&\lesssim_p d^p_{\psi_2}(\F)\bigg(\frac{\bar{\gamma}(\F,\psi_2)}{\sqrt{N}}+\frac{\big((\bar{\gamma}(\F,\psi_2)+(\log N)^{1/2}\big)^p}{N}\bigg) \nonumber\\
&\asymp_p \frac{\gamma(\F,\psi_2)d^{p-1}_{\psi_2}(\F)}{\sqrt{N}}+\frac{\gamma^{p}(\F,\psi_2)}{N},
\end{align}
where the final equivalence follows from the bound $(\log N)^{p/2}/N\lesssim_p 1/\sqrt{N}$ and $\bar{\gamma}(\F,\psi_2)\gtrsim 1$ (see \cite[Lemma 4.6]{al2025sharp}). This recovers the sharp expectation bound for the supremum of $L_p$ empirical processes established in \cite[Theorem 2.2]{al2025sharp}. We refer the reader to \cite{al2025sharp} for further discussion of \eqref{eq:rem_aux2}, as well as additional references to related work concerning quadratic and product empirical processes.

Similar to the discussion in Remark \ref{rem:main1_1}, the logarithmic factors in the upper bound of Theorem \ref{thm:main2} are, in general, unavoidable. This new phenomenon arises only in multi-product empirical processes involving at least three components in the product (i.e. $p\ge 3$), and when the complexity parameters $\bar{\gamma}(\F^{(k)},\psi_2)$ of the associated function classes vary significantly  ---specifically, when some exceed $(\log N)^{1/2}$ while others fall below this threshold. Consequently, such behavior has not been observed in previous work on the theory of quadratic and product empirical processes.
\end{remark}

\section{Concentration of simple random tensors}\label{sec:concentration}

This section contains the proof of our first main result, Theorem \ref{thm:main1}. The upper bound for sub-Gaussian data follows from an application of Theorem \ref{thm:main2}, where for each $1\le k\le p$, the class $\F^{(k)}$ consists of all linear functionals in $H^{(k)}$ with norm one. To establish the matching lower bound for independent Gaussian data, we adapt techniques from \cite[Theorem 4]{koltchinskii2017concentration} to the multi-product setting and make a key observation that precisely captures the logarithmic factors appearing in the upper bound; see Proposition \ref{prop:lowerbound} for the proof and Remark \ref{rem:lower_bound} for further discussion.

\begin{proof}[Proof of Theorem \ref{thm:main1}]

We have
 \begin{align*}
\bigg\|\frac{1}{N}\sum_{i=1}^{N} X^{(1)}_i\otimes \cdots \otimes X^{(p)}_{i} -\E\, X^{(1)}\otimes \cdots\otimes X^{(p)} \bigg\| &= \sup_{\|v_1\|=\cdots =\|v_p\|=1} \bigg|\frac{1}{N}\sum_{i=1}^N \prod_{k=1}^{p} \langle X^{(k)}_i, v_k\rangle- \E \prod_{k=1}^{p} \langle X^{(k)} , v_k \rangle \bigg|\\
&=\sup_{f^{(k)} \in \mathcal{F}^{(k)},1\le k\le p}\bigg|\frac{1}{N} \sum_{i=1}^N \prod_{k=1}^p f^{(k)} (X^{(k)}_i)-\mathbb{E} \prod_{k=1}^p f^{(k)} (X^{(k)})\bigg|,
 \end{align*}
 where $\F^{(k)}:=\{\langle\cdot, v\rangle: v \in U_{H^{(k)}}\}, U_{H^{(k)}}:=\{v \in H^{(k)}:\|v\| =1\}$ for each $1\le k\le p$. It is clear that $\F^{(k)}$ is symmetric. Recall that for $1\le k\le p$, $X^{(k)}, X_1^{(k)}, \ldots, X_N^{(k)}$ are i.i.d. centered sub-Gaussian and pre-Gaussian random variables in $H^{(k)}$ with covariance operator $\Sigma^{(k)}$. Since $X^{(k)}$ is sub-Gaussian, the $\psi_2$-norm of linear functionals is equivalent to the $L_2$-norm. Hence,
\begin{align*}\label{eq:aux1}
d_{\psi_2}(\F^{(k)})=\sup_{f^{(k)}\in\mathcal{F}^{(k)}}\|f^{(k)}\|_{\psi_2}\asymp \sup_{f^{(k)}\in\mathcal{F}^{(k)}}\|f^{(k)}\|_{L_2}=\sup_{\|v\|= 1}\big(\E\langle X^{(k)},v\rangle^2\big)^{1/2}=\|\Sigma^{(k)}\|^{1/2}.
\end{align*}
Moreover, since $X^{(k)}$ is pre-Gaussian, there exists a centered Gaussian random variable $Y^{(k)}$ in $H^{(k)}$ with the same covariance $\Sigma^{(k)}$. This means that, for $u, v \in U_{H^{(k)}}$, the canonical metric associated with $Y^{(k)}$ satisfies
\[
d_{Y^{(k)}}(u, v)
= \big(\mathbb{E}(\langle Y^{(k)}, u\rangle-\langle Y^{(k)}, v\rangle)^2\big)^{1/2}=\big(\langle u-v, \Sigma^{(k)} (u-v) \rangle\big)^{1/2}
=\|\langle\cdot, u\rangle-\langle\cdot, v\rangle\|_{L_2(\mu^{(k)})},
\]
where $\mu^{(k)}$ is the law of $X^{(k)}$. Using Talagrand's majorizing-measure theorem \cite[Theorem 2.10.1]{talagrand2022upper}, we deduce that
\[
\gamma(\mathcal{F}^{(k)}, \psi_2)\asymp \gamma(\mathcal{F}^{(k)}, L_2)=\gamma\big(U_{H^{(k)}}, d_{Y^{(k)}}\big) \asymp \mathbb{E} \sup _{v \in U_{H^{(k)}}}\langle Y^{(k)}, v\rangle = \E \|Y^{(k)}\| \asymp \big(\E \|Y^{(k)}\|^2\big)^{1/2}=\mathrm{Tr}(\Sigma^{(k)})^{1/2}.
\]
Therefore,
 \[
 d_{\psi_2}(\F^{(k)}) \asymp \|\Sigma^{(k)}\|^{1/2},\quad \gamma(\F^{(k)},\psi_2)\asymp \mathrm{Tr}(\Sigma^{(k)})^{1/2},\qquad 1\le k\le p.
 \]
Consequently, the upper bound stated in Theorem \ref{thm:main1} follows directly from Theorem \ref{thm:main2}. The matching lower bound for independent Gaussian data is proved in Proposition \ref{prop:lowerbound} below (the case $q=1$). Moreover, for any $q\ge 1$,
\[
\left(\E \bigg\|\frac{1}{N} \sum_{i=1}^N X_i^{(1)} \otimes \cdots \otimes X_i^{(p)}-\E\, X^{(1)} \otimes \cdots \otimes X^{(p)}\bigg\|^q \right)^{1/q}\ge \E \bigg\|\frac{1}{N} \sum_{i=1}^N X_i^{(1)} \otimes \cdots \otimes X_i^{(p)}-\E\, X^{(1)} \otimes \cdots \otimes X^{(p)}\bigg\|,
\]
so the first-moment lower bound immediately implies the same lower bound for all $q$-th moments.
\end{proof}

\begin{proposition}\label{prop:lowerbound} For any integer $p\ge 2$ and $1 \leq k \leq p$, let $X^{(k)}, X_1^{(k)}, \ldots, X_N^{(k)}$ be i.i.d. centered Gaussian random variables in $H^{(k)}$ with covariance operator $\Sigma^{(k)}$, and suppose that $X^{(1)},\ldots, X^{(p)}, (X^{(1)}_i)_{i=1}^{N},\ldots,(X^{(p)}_i)_{i=1}^N,$ are independent. Then,
\[
\E \bigg\|\frac{1}{N} \sum_{i=1}^N X_i^{(1)} \otimes \cdots \otimes X_i^{(p)}-\E\, X^{(1)} \otimes \cdots \otimes X^{(p)}\bigg\| \gtrsim_p\bigg(\prod_{k=1}^p\|\Sigma^{(k)}\|^{1 / 2}\bigg)\mathscr{E}_N\big((\Sigma^{(k)})_{k=1}^p\big),
\]
where
\[
\mathscr{E}_N\big((\Sigma^{(k)})_{k=1}^p\big):= \bigg(\frac{\sum_{k=1}^p r(\Sigma^{(k)})}{N}\bigg)^{1/2}+\frac{1}{N}\prod_{k= 1}^{p}\Big(r(\Sigma^{(k)})+\log N\Big)^{1/2}.
\]
\end{proposition}

\begin{proof}[Proof of Proposition \ref{prop:lowerbound}]
First, we may assume without loss of generality that the effective ranks satisfy $r(\Sigma^{(1)})\ge r(\Sigma^{(2)})\ge \cdots\ge r(\Sigma^{(p)})$. It is straightforward to check that
    \[
\mathscr{E}_N\big((\Sigma^{(k)})_{k=1}^p\big)\asymp_p \begin{cases}
    \Big(\frac{\sum_{k=1}^p r(\Sigma^{(k)})}{N}\Big)^{1/2}+\frac{(\log N)^{(p-t)/2}}{N}\prod_{k=1}^{t}r(\Sigma^{(k)})^{1/2} & \mathrm{ if } \ \exists \ t: r(\Sigma^{(t)})\ge \log N\ge r(\Sigma^{(t+1)}), \\
  \Big(\frac{\sum_{k=1}^p r(\Sigma^{(k)})}{N}\Big)^{1/2}  & \text{otherwise}. \\
\end{cases}
    \]

The conclusion follows directly from the following two claims:

\textbf{Claim I.} It holds that
\begin{align*}
\E \bigg\|\frac{1}{N} \sum_{i=1}^N X_i^{(1)} \otimes \cdots \otimes X_i^{(p)}-\E\, X^{(1)} \otimes \cdots \otimes X^{(p)}\bigg\| \gtrsim_p\bigg(\prod_{k=1}^p\|\Sigma^{(k)}\|^{1 / 2}\bigg) \bigg(\frac{\sum_{k=1}^p r(\Sigma^{(k)})}{N}\bigg)^{1/2}.
\end{align*}

\textbf{Claim II.}
If there exists an index $t$ such that $r(\Sigma^{(t)})\ge \log N \ge r(\Sigma^{(t+1)})$, then
\begin{align*}
\E \bigg\|\frac{1}{N} \sum_{i=1}^N X_i^{(1)} \otimes \cdots \otimes X_i^{(p)}-\E\, X^{(1)} \otimes \cdots \otimes X^{(p)}\bigg\| \gtrsim_p\bigg(\prod_{k=1}^p\|\Sigma^{(k)}\|^{1 / 2}\bigg) \frac{(\log N)^{(p-t)/2}}{N}\prod_{k=1}^{t}r(\Sigma^{(k)})^{1/2}.
\end{align*}

Now we prove \textbf{Claim I} and \textbf{Claim II}.

\textbf{Proof of Claim I:}
By the assumption of independence and the definition of the tensor norm, we obtain
\begin{align*}
&\E \bigg\|\frac{1}{N} \sum_{i=1}^N X_i^{(1)} \otimes \cdots \otimes X_i^{(p)}-\E\, X^{(1)} \otimes \cdots \otimes X^{(p)}\bigg\|=\E \bigg\|\frac{1}{N} \sum_{i=1}^N X_i^{(1)} \otimes \cdots \otimes X_i^{(p)}\bigg\| \\
&= \E \sup_{\|v_1\|= \cdots =\|v_p\|=1} \bigg| \frac{1}{N}\sum_{i=1}^N \langle X^{(1)}_i,v_1\rangle \cdots \langle X^{(p)}_i, v_p \rangle \bigg|\\
&= \E \sup_{\|v_2\|= \cdots =\|v_p\|=1}\bigg\| \frac{1}{N}\sum_{i=1}^N \langle X^{(2)}_i,v_2\rangle \cdots \langle X^{(p)}_i, v_p \rangle X_i^{(1)} \bigg\|\\
&=\E_{(X^{(k)}_i)_{i=1}^{N}, 2\le k\le p} \left[ \E_{(X^{(1)}_i)_{i=1}^{N}}  \sup_{\|v_2\|= \cdots =\|v_p\|=1}\bigg\| \frac{1}{N}\sum_{i=1}^N \langle X^{(2)}_i,v_2\rangle \cdots \langle X^{(p)}_i, v_p \rangle X_i^{(1)} \bigg\|\right].
\end{align*}
Note that, conditionally on $\langle X^{(k)}_i, v_k \rangle, 1\le i\le  N, 2\le k\le p$, the distribution of the random variable
\[
\frac{1}{N}\sum_{i=1}^N \langle X^{(2)}_i,v_2\rangle \cdots \langle X^{(p)}_i, v_p \rangle X_i^{(1)}
\]
is Gaussian and it coincides with the distribution of the random variable
\[
\bigg(\frac{1}{N} \sum_{i=1}^{N}\langle X^{(2)}_i, v_2\rangle^{2}\cdots \langle X^{(p)}_i, v_p\rangle^{2} \bigg)^{1 / 2} \frac{X^{(1)}}{\sqrt{N}}.
\]
Therefore,
\begin{align*}
    &\E_{(X^{(k)}_i)_{i=1}^{N}, 2\le k\le p} \left[ \E_{(X^{(1)}_i)_{i=1}^{N}}  \sup_{\|v_2\|= \cdots =\|v_p\|=1}\bigg\| \frac{1}{N}\sum_{i=1}^N \langle X^{(2)}_i,v_2\rangle \cdots \langle X^{(p)}_i, v_p \rangle X_i^{(1)} \bigg\|\right]\\
    &=\E_{(X^{(k)}_i)_{i=1}^{N}, 2\le k\le p} \left[ \sup_{\|v_2\|= \cdots =\|v_p\|=1}\bigg(\frac{1}{N} \sum_{i=1}^{N}\langle X^{(2)}_i, v_2\rangle^{2}\cdots \langle X^{(p)}_i, v_p\rangle^{2} \bigg)^{1 / 2} \frac{\E \|X^{(1)}\|}{\sqrt{N}}\right]\\
    &=\frac{\E \|X^{(1)}\|}{\sqrt{N}}\E_{(X^{(k)}_i)_{i=1}^{N}, 2\le k\le p}\left[ \sup_{\|v_2\|= \cdots =\|v_p\|=1}\bigg(\frac{1}{N} \sum_{i=1}^{N}\langle X^{(2)}_i, v_2\rangle^{2}\cdots \langle X^{(p)}_i, v_p\rangle^{2} \bigg)^{1 / 2} \right]\\
    &\ge \frac{\E \|X^{(1)}\|}{\sqrt{N}} \sup_{\|v_2\|= \cdots =\|v_p\|=1}\E_{(X^{(k)}_i)_{i=1}^{N}, 2\le k\le p}\bigg(\frac{1}{N} \sum_{i=1}^{N}\langle X^{(2)}_i, v_2\rangle^{2}\cdots \langle X^{(p)}_i, v_p\rangle^{2} \bigg)^{1 / 2}\\
    &=\frac{\E \|X^{(1)}\|}{\sqrt{N}}  \sup_{\|v_2\|= \cdots =\|v_p\|=1} \langle \Sigma^{(2)} v_2,v_2 \rangle^{1/2}\cdots \langle \Sigma^{(p)} v_p,v_p \rangle^{1/2} \E \bigg(\frac{1}{N}\sum_{i=1}^{N} (Z_i^{(2)})^{2} \cdots (Z_i^{(p)})^{2} \bigg)^{1/2}\\
    &=\frac{\E \|X^{(1)}\|}{\sqrt{N}}\bigg(\prod_{k=2}^p\|\Sigma^{(k)}\|^{1 / 2}\bigg)\E \bigg(\frac{1}{N}\sum_{i=1}^{N} (Z_i^{(2)})^{2} \cdots (Z_i^{(p)})^{2} \bigg)^{1/2},
\end{align*}
where
\[
Z^{(k)}_i :=\frac{\langle X^{(k)}_i,v_k \rangle}{\langle \Sigma^{(k)} v_k ,v_k\rangle^{1/2}},\quad 1\le i\le N, \ 1\le k\le p,
\]
are i.i.d. standard Gaussian random variables. One can check that
\[
\E\bigg(\frac{1}{N}\sum_{i=1}^{N} (Z_i^{(2)})^{2} \cdots (Z_i^{(p)})^{2} \bigg)^{1/2}\ge c(p)
\]
for a positive constant $c(p)$ depending only on $p$, which implies that
\begin{align*}
    &\E \bigg\|\frac{1}{N} \sum_{i=1}^N X_i^{(1)} \otimes \cdots \otimes X_i^{(p)}-\E\, X^{(1)} \otimes \cdots \otimes X^{(p)}\bigg\|\\
    &\ge \frac{\E \|X^{(1)}\|}{\sqrt{N}}\bigg(\prod_{k=2}^p\|\Sigma^{(k)}\|^{1 / 2}\bigg)\E \bigg(\frac{1}{N}\sum_{i=1}^{N} (Z_i^{(2)})^{2} \cdots (Z_i^{(p)})^{2} \bigg)^{1/2}\\
    &\gtrsim_p \bigg(\prod_{k=1}^p\|\Sigma^{(k)}\|^{1 / 2}\bigg)\bigg(\frac{r(\Sigma^{(1)})}{N}\bigg)^{1/2}
    \gtrsim_p \bigg(\prod_{k=1}^p\|\Sigma^{(k)}\|^{1 / 2}\bigg)\bigg(\frac{\sum_{k=1}^p r(\Sigma^{(k)})}{N}\bigg)^{1/2},
\end{align*}
where in the last step we used that $r(\Sigma^{(1)})\ge r(\Sigma^{(2)})\ge \cdots\ge r(\Sigma^{(p)})$. This establishes \textbf{Claim I}.

\textbf{Proof of Claim II:}
Suppose that there exists an index $t$ such that 
\[
r(\Sigma^{(1)})\ge \cdots \ge r(\Sigma^{(t)})\ge \log N \ge r(\Sigma^{(t+1)})\ge \cdots \ge r(\Sigma^{(p)}).
\]
Recall that
\begin{align*}
&\E \bigg\|\frac{1}{N} \sum_{i=1}^N X_i^{(1)} \otimes \cdots \otimes X_i^{(p)}-\E\, X^{(1)} \otimes \cdots \otimes X^{(p)}\bigg\|=\E \bigg\|\frac{1}{N} \sum_{i=1}^N X_i^{(1)} \otimes \cdots \otimes X_i^{(p)}\bigg\| \\
&= \E \sup_{\|v_1\|= \cdots =\|v_p\|=1} \bigg| \frac{1}{N}\sum_{i=1}^N \langle X^{(1)}_i,v_1\rangle \cdots \langle X^{(p)}_i, v_p \rangle \bigg|.
\end{align*}

If $t=p$, we have
\begin{align*}
   &\E \sup_{\|v_1\|= \cdots =\|v_p\|=1} \bigg| \frac{1}{N}\sum_{i=1}^N \langle X^{(1)}_i,v_1\rangle \cdots \langle X^{(p)}_i, v_p \rangle \bigg|\\
   &\ge \E_{(X_1^{(k)})_{k=1}^{p}} \sup_{\|v_1\|= \cdots =\|v_p\|=1}  \bigg|\E_{(X_i^{(k)})_{i=2}^{N},1\le k\le p}\frac{1}{N}\sum_{i=1}^N \langle X^{(1)}_i,v_1\rangle \cdots \langle X^{(p)}_i, v_p \rangle \bigg|\\
   &=\E_{(X_1^{(k)})_{k=1}^{p}} \sup_{\|v_1\|= \cdots =\|v_p\|=1}  \bigg| \frac{1}{N} \langle X^{(1)}_1,v_1\rangle \cdots \langle X^{(p)}_1, v_p \rangle \bigg|\\
   &=\frac{1}{N}\E_{(X_1^{(k)})_{k=1}^{p}}\prod_{k=1}^{p}\|X_1^{(k)}\| \overset{\text{($\star$)}}{=} \frac{1}{N}\prod_{k=1}^{p}\E \|X_1^{(k)}\|  \gtrsim_p \frac{1}{N}  \bigg(\prod_{k=1}^p\|\Sigma^{(k)}\|^{1 / 2}\bigg) \prod_{k=1}^{p}r(\Sigma^{(k)})^{1/2},
\end{align*}
as desired. Here $(\star)$ holds since $X^{(1)}_1,\ldots, X^{(p)}_1$ are independent, so $\|X^{(1)}_1\|,\ldots, \|X^{(p)}_1\|$ are independent and the expectation factorizes.

If $1 \le t \le p-1$, then by choosing $v_k=v^{*}_k$ to be the unit eigenvector associated with the largest eigenvalue of $\Sigma^{(k)}$, for each $k \ge t+1$, we obtain 
\begin{align}\label{eq:lowerbound_aux1}
&\E \sup_{\|v_1\|= \cdots =\|v_p\|=1} \bigg| \frac{1}{N}\sum_{i=1}^N \langle X^{(1)}_i,v_1\rangle \cdots \langle X^{(p)}_i, v_p \rangle \bigg|\nonumber\\
    &\ge \E \sup_{\|v_1\|= \cdots =\|v_t\|=1} \bigg| \frac{1}{N}\sum_{i=1}^N \langle X^{(1)}_i,v_1\rangle \cdots \langle X^{(t)}_i, v_t \rangle \langle X^{(t+1)}_i, v^{*}_{t+1} \rangle\cdots \langle X^{(p)}_i, v^{*}_p \rangle\bigg| \nonumber\\
    &= \frac{1}{N}\bigg(\prod_{k=t+1}^p\|\Sigma^{(k)}\|^{1 / 2}\bigg) \E \sup_{\|v_1\|= \cdots =\|v_t\|=1} \bigg| \sum_{i=1}^N \langle X^{(1)}_i,v_1\rangle \cdots \langle X^{(t)}_i, v_t \rangle Z^{(t+1)}_i\cdots Z^{(p)}_i \bigg|,
\end{align}
where \[
Z^{(k)}_i =\frac{\langle X^{(k)}_i,v^*_k \rangle}{\langle \Sigma^{(k)} v^*_k ,v^*_k\rangle^{1/2}}=\frac{\langle X^{(k)}_i,v^*_k \rangle}{\|\Sigma^{(k)}\|^{1/2}},\quad 1\le i\le N, \ 1\le k\le p,
\]
are i.i.d. standard Gaussian random variables.

Observe that
\begin{align}\label{eq:lowerbound_aux2}
    &\E \sup_{\|v_1\|= \cdots =\|v_t\|=1} \bigg| \sum_{i=1}^N \langle X^{(1)}_i,v_1\rangle \cdots \langle X^{(t)}_i, v_t \rangle Z^{(t+1)}_i\cdots Z^{(p)}_i \bigg| \nonumber\\
    &= \E_{(Z^{(k)}_i)_{i=1}^{N},k\ge t+1}\left[\E_{(X^{(k)}_i)_{i=1}^{N},k\le t}  \sup_{\|v_1\|= \cdots =\|v_t\|=1} \bigg| \sum_{i=1}^N \langle X^{(1)}_i,v_1\rangle \cdots \langle X^{(t)}_i, v_t \rangle Z^{(t+1)}_i\cdots Z^{(p)}_i \bigg| \right] \nonumber\\
    &\ge \E_{(Z^{(k)}_i)_{i=1}^{N},k\ge t+1}\left[\max_{1\le j\le N}\E_{(X^{(k)}_j)_{k=1}^{t}} \sup_{\|v_1\|= \cdots =\|v_t\|=1} \bigg|  \E_{(X^{(k)}_i)_{k=1}^{t},i\ne j} \sum_{i=1}^N \langle X^{(1)}_i,v_1\rangle \cdots \langle X^{(t)}_i, v_t \rangle Z^{(t+1)}_i\cdots Z^{(p)}_i \bigg|\right] \nonumber\\
    &=\E_{(Z^{(k)}_i)_{i=1}^{N},k\ge t+1}\left[\max_{1\le j\le N}\E_{(X^{(k)}_j)_{k=1}^{t}} \sup_{\|v_1\|= \cdots =\|v_t\|=1} \bigg| \langle X^{(1)}_j,v_1\rangle \cdots \langle X^{(t)}_j, v_t \rangle Z^{(t+1)}_j\cdots Z^{(p)}_j \bigg|\right] \nonumber\\
    &\ge \E_{(Z^{(k)}_i)_{i=1}^{N},k\ge t+1}\left[\max_{1\le j\le N}\E_{(X^{(k)}_j)_{k=1}^{t}}   \| X^{(1)}_j \| \cdots \| X^{(t)}_j \|
|Z^{(t+1)}_j| \cdots |Z^{(p)}_j| \right] \nonumber\\
&= \E_{(Z^{(k)}_i)_{i=1}^{N},k\ge t+1}\left[\max_{1\le j\le N} (\E \| X^{(1)} \|) \cdots (\E \| X^{(t)} \|)
|Z^{(t+1)}_j| \cdots |Z^{(p)}_j| \right] \nonumber\\
&= \bigg(\prod_{k=1}^{t } \E \| X^{(k)} \|\bigg)\E_{(Z^{(k)}_i)_{i=1}^{N},k\ge t+1} \left[\max_{1\le j\le N} |Z^{(t+1)}_j| \cdots |Z^{(p)}_j| \right] \nonumber\\
&\asymp_p \bigg(\prod_{k=1}^{t } \|\Sigma^{(k)}\|^{1/2}\bigg)\bigg(\prod_{k=1}^{t } r(\Sigma^{(k)})^{1/2}\bigg) (\log N)^{(p-t)/2},
\end{align}
where in the last step we use
Lemma \ref{lemma:sup_Gaussian_products} from Appendix \ref{app:A}, together with the fact that the random variables $Z^{(k)}_i$, for $1\le i\le N$ and $t+1\le k\le p$, are i.i.d. standard Gaussian. Combining \eqref{eq:lowerbound_aux1} and \eqref{eq:lowerbound_aux2} completes the proof of \textbf{Claim II}.   
\end{proof}

\begin{remark}[Discussion on Proposition \ref{prop:lowerbound}]\label{rem:lower_bound}
The conditioning idea used in the proof of \emph{\textbf{Claim I}} can be viewed as a multi-product analogue of the lower-bound argument in \cite[Theorem~4]{koltchinskii2017concentration}, where it is shown that for i.i.d.\ centered Gaussian random variables $X,X_1,\ldots,X_N$ with covariance operator $\Sigma$,
\[
\E \bigg\|\frac{1}{N}\sum_{i=1}^{N} X_i \otimes X_i- \E\, X \otimes X\bigg\|
\;\gtrsim\;
\|\Sigma\|\sqrt{\frac{r(\Sigma)}{N}}.
\]
In our multi-product setting, the key difference is that the tensor factors may correspond to different random variables. We therefore begin by focusing on the component $X^{(1)}$ whose covariance has the largest effective rank. Under independence across components, the argument reduces to a clean conditioning step that isolates the contribution of this dominant factor.

The proof of \emph{\textbf{Claim II}} is where the additional logarithmic factors in the asymmetric case arise. The idea is to choose test vectors as follows: for those components whose covariance effective ranks are smaller than $\log N$, we take the unit eigenvector corresponding to the largest eigenvalue and denote the resulting normalized scalar variables by $(Z_i^{(k)})_{i=1}^{N}$ for $ k\ge t+1$. Conditioning on $\{Z_i^{(k)}:1\le i\le N,k\ge t+1\}$, we can move the expectation inside the absolute value to obtain a lower bound, and, crucially, eliminate all but one of the $N$ summands. More precisely, we essentially keep only the index $j$ corresponding to the largest fluctuation of
\[
|Z_j^{(t+1)}|\cdots |Z_j^{(p)}|.
\]
Since there are  $N$ i.i.d. copies and the components are independent, this maximum is typically of order $(\log N)^{(p-t)/2}$, yielding the sharp logarithmic contribution. In this way, the factors with larger effective ranks contribute through their effective-rank terms, while the smaller-rank factors contribute through the maximal fluctuation across the $N$ i.i.d. draws. By contrast, a direct approach that simply moves expectation into the absolute value without this selection would only yield a lower bound of order
\[
\frac{1}{N}\bigg(\prod_{k=1}^p\|\Sigma^{(k)}\|^{1/2}\bigg)\prod_{k=1}^{p} r(\Sigma^{(k)})^{1/2},
\]
which corresponds to the ``no-log" case $t=p$. The above extremal-selection observation is instrumental to obtain the sharper lower bound that matches our upper bound.

Finally, we emphasize that we assume independence across tensor components only to provide a clean special case in which our general upper bound is provably tight, without imposing any additional structural assumptions on the correlations between components. In particular, the lower-bound mechanism in the proof makes clear that there are many settings beyond full independence in which the same kind of extremal-fluctuation phenomenon can occur, and hence where the resulting logarithmic factors may still appear and be unavoidable.
\end{remark}

\section{Multi-product empirical processes}\label{sec:multi}

This section studies multi-product empirical processes indexed by $p$ possibly different function classes $(\F^{(k)})_{k=1}^{p}$, evaluated at $p$ possibly different random variables $(X^{(k)})_{k=1}^{p}$:
\begin{align}\label{eq:multi-product-processes}
(f^{(1)},\ldots,f^{(p)})\mapsto \frac{1}{N}\sum_{i=1}^N \prod_{k=1}^p f^{(k)}(X^{(k)}_i)-\E \prod_{k=1}^p f^{(k)}(X^{(k)}),\qquad f^{(k)}\in \F^{(k)}, \ 1\le k\le p.
\end{align}
We leverage generic chaining techniques to derive sharp upper bounds on their suprema in terms of quantities that reflect the complexity of the function classes $(\F^{(k)})_{k=1}^{p}$.

Our theory builds on the framework introduced by Mendelson in \cite{mendelson2016upper}. The central idea is as follows. First, we gather precise structural information on the typical coordinate projection of $\F^{(k)}, 1\le k\le p$; that is, for $\sigma^{(k)}=(X^{(k)}_1,\ldots,X^{(k)}_N), 1\le k\le p$, we study the structure of
\[
P_{\sigma^{(k)}}(\F^{(k)})=\left\{(f^{(k)}(X^{(k)}_i))_{i=1}^{N}: f^{(k)}\in \F^{(k)}\right\},\quad 1 \le k\le p.
\]
In order to study multi-product processes with $p\ge 3$, it is further necessary to understand the structure of
\[
\bigg(\prod_{k'\ne k} f^{(k')}(X_i^{(k')})\bigg)_{i=1}^{N},\quad 1\le k\le p,
\]
a structural component that appears to play a central role in the multi-product setting considered here. Second, for typical $\sigma^{(k)}=(X^{(k)}_1,\ldots,X^{(k)}_N), 1\le k\le p$, we analyze the suprema of the conditioned Bernoulli process:
\[
(f^{(1)},\ldots,f^{(p)}) \mapsto \sum_{i=1}^N \varepsilon_i \prod_{k=1}^{p}f^{(k)}(X^{(k)}_i)  \  \Big | \sigma^{(1)},\ldots ,\sigma^{(p)},
\]
where $\varepsilon_1,\ldots,\varepsilon_N$ are independent symmetric Bernoulli random variables (that is, $\P(\varepsilon_n=+1)=\P(\varepsilon_n=-1)=1/2$).

The structure of typical coordinate projections of sub-Gaussian function classes has been extensively studied in empirical process theory \cite{mendelson2010empirical, mendelson2011discrepancy,mendelson2012generic,mendelson2016dvoretzky,mendelson2016upper}, primarily in the context of quadratic and product empirical processes. Here, we extend the analysis to empirical processes indexed by $p$ different function classes evaluated at $p$ different random variables. A key ingredient of our approach is Theorem \ref{thm:key_lemma}, which establishes a high-probability upper bound on the supremum of the $\ell_2$-norm of the product of coordinate projection vectors of the function classes $(\F^{(k)})_{k=1}^{p}$. Specifically, it controls
\[
\sup_{f^{(k)}\in \F^{(k)},1\le k\le p} \bigg(\sum_{i=1}^{N}\prod_{k=1}^{p} \left(f^{(k)}(X_i^{(k)})\right)^2\bigg)^{1/2}
\]
in terms of quantities that capture the geometric complexity of $(\F^{(k)})_{k=1}^{p}$. A novel phenomenon captured by Theorem \ref{thm:key_lemma} is the emergence of logarithmic factors in the multi-product setting; these factors do not appear in the analysis of quadratic or standard product empirical processes.

The rest of this section is organized as follows. Subsection \ref{subsec:background} introduces some definitions from Mendelson's work \cite{mendelson2016upper} and discusses an important result, Lemma \ref{lemma:mendelson} (\cite[Corollary 3.6]{mendelson2016upper}), which plays a crucial role in our proof of Theorem \ref{thm:main2}. We then prove Theorem \ref{thm:main2} in Subsection \ref{subsec:main2}. Finally, Subsection \ref{subsec:prooflemma} contains the proof of Theorem \ref{thm:key_lemma} and Subsection \ref{subsec:log-concave} discusses how to extend our theory to log-concave ensembles. 

\subsection{Background and auxiliary results}\label{subsec:background}
Following \cite{mendelson2016upper}, we will utilize complexity parameters that take into account all the $L_p$ structures endowed by the process. We recall that, for $q \geq 1,$ the graded $L_q$-norm is defined by
\[
\|f\|_{(q)} :=\sup _{1 \leq p \leq q} \frac{\|f\|_{L_p}}{\sqrt{p}}.
\]
We will use repeatedly that, for any $q \ge 2,$ $\|f \|_{L_2} \lesssim \|f\|_{(q)} \lesssim \|f\|_{\psi_2}.$
As shown in \cite{mendelson2016upper}, the graded $\gamma$-type functionals that we now define serve as key parameters to characterize the complexity of the typical coordinate projection of a function class.

\begin{definition}[Graded $\gamma$-type functionals]
Given a class of functions $\F$ on a probability space $(\Omega, \mu)$, constants $u \geq 1$ and $s_0 \geq 0$, set
\[
\Lambda_{s_0, u}(\F) :=\inf \sup _{f \in \F} \sum_{s \geq s_0} 2^{s / 2}\left\|f-\pi_s f\right\|_{(u^2 2^s)},
\]
where the infimum is taken over all admissible sequences, and $\pi_s f$ is the nearest point in $\F_s$ to $f$ with respect to the $(u^2 2^s)$-norm. Also set
\[
\widetilde{\Lambda}_{s_0, u}(\F) :=\Lambda_{s_0, u}(\F)+2^{s_0 / 2} \sup_{f \in \F}\left\|\pi_{s_0} f\right\|_{(u^2 2^{s_0})} .
\]
\end{definition} 

 The following lemma shows that $\Lambda_{s_0, u}(\F)$ is upper bounded by $\gamma(\F,\psi_2)$ up to a universal constant.
\begin{lemma}[{\cite[Lemma 4.1]{al2025sharp}}]\label{lemma:Lambda_Gamma}
    For any $u\ge 1$ and $s_0\ge 0$, $\Lambda_{s_0,u}(\F)\lesssim \gamma(\F,\psi_2)$.
\end{lemma}

To analyze (conditioned) Bernoulli processes, we will use a chaining argument combined with the following H\"offding’s inequality.

\begin{lemma}[H\"offding's inequality, \cite{montgomery1990distribution}]\label{lemma:Hoffding} Let $\varepsilon_1,\ldots,\varepsilon_N$ be independent symmetric Bernoulli random variables. For any fixed $I \subset\{1, \ldots, N\}$ and $z \in \mathbb{R}^N$, it holds with probability at least $1-2 \exp (-t^2 / 2 )$ that
\[
\bigg|\sum_{i=1}^N \varepsilon_i z_i\bigg| \le \sum_{i \in I}\left|z_i\right|+t\bigg(\sum_{i \in I^c} z_i^2\bigg)^{1 / 2}.
\]
In particular, for every $1 \le k \le N,$ it holds with probability at least $1-2 \exp (-t^2 / 2 )$ that
\[
\bigg|\sum_{i=1}^N \varepsilon_i z_i\bigg| 
\le \sum_{i=1}^k \left|z_i^*\right|+t\bigg(\sum_{i=k+1}^N (z_i^*)^2\bigg)^{1 / 2},
\]
where $\left(z_i^*\right)_{i=1}^N$ denotes the non-increasing rearrangement of $\left(\left|z_i\right|\right)_{i=1}^N.$ This estimate is optimal when $k\asymp t^2$.
\end{lemma}

According to Lemma \ref{lemma:Hoffding}, the influence of $z$ on the Bernoulli process depends on the $\ell_1$-norm of the largest $k$ coordinates of $(|z_i|)_{i=1}^{N}$ and the $\ell_2$-norm of the smallest $N-k$ coordinates. As noted in \cite{mendelson2016upper}, the graded $L_q$-norm plays a key role in analyzing the monotone non-increasing rearrangement of $N$ independent copies of a random variable. In particular, \cite[Corollary 3.6]{mendelson2016upper} provides a useful way to decompose a collection of random vectors into their largest and smallest coordinate components, a technique central to \cite{mendelson2016upper}. Below, we present the statement obtained by setting $q=2r, \beta=1/2$, and replacing $s$ with $s+\lceil \log_2 p \rceil$ in their notation.

\begin{lemma}[{\cite[Corollary 3.6]{mendelson2016upper}}]\label{lemma:mendelson}

There exist absolute constants $c_0, c_1$ and for every $r \ge 1$ there exist constants $c_2, c_3, c_4$ and $c_5$ that depend only on $r$ for which the following holds. Set
\begin{align}\label{eq:j_s}
j_s =\min \left\{\left\lceil
\frac{c_0 u^2 p 2^s}{\log \left(4+e N / u^2 p2^s\right)}\right\rceil, N+1\right\}
\end{align}
for $u \geq c_2$. Let $\mathcal{H} \subset L_{2r}$  be of cardinality at most $2^{p2^{s+2}}$. For any $h\in \mathcal{H}$, the random vector $\left(h\left(X_i\right)\right)_{i=1}^N$ can be decomposed into a sum $U_h+V_h$, where the random vectors $U_h,V_h \in \R^N$ have disjoint supports. Namely, $U_h$ is supported on the largest $j_s-1$ coordinates of $\left(\left|h\left(X_i\right)\right|\right)_{i=1}^N$, while $V_h$ is supported on the complement. Then, there is an event of probability at least $1-2 \exp \left(-c_3 u^2 p 2^s\right)$ on which, for every $h \in \mathcal{H}$,
$$
\left\|U_h\right\|_{\ell_2^N} \leq c_4 u p^{1/2}2^{s / 2}\|h\|_{(c_5 u^2 p 2^s)}, \quad  
\quad\left\|V_h\right\|_{\ell_r^N} \leq c_1 \|h\|_{L_{2r}} N^{1 / r}.
$$
\end{lemma}

For $1\le k\le p$, let $(\F^{(k)}_s)_{s \geq 0}$ be an admissible sequence of $\F^{(k)}$. For $f^{(k)}\in \F^{(k)}$ and $s\ge 0$, we define $\Delta_s f^{(k)} := \pi_{s+1}f^{(k)}-\pi_sf^{(k)}$ and $(\Delta_s f^{(k)})_i := (\Delta_s f^{(k)})(X^{(k)}_i)$ for $1\le i\le N$. To analyze the multi-product empirical process in \eqref{eq:multi-product-processes}, we will apply Lemma \ref{lemma:mendelson} to the classes $\F^{(k)}_s$, $\{\Delta_s f^{(k)}:f^{(k)}\in \F^{(k)}\}$ for each $k$, as well as to the classes consisting of products of coordinate projection vectors, with specifically chosen values of $r$, leading to the following corollary.

\begin{corollary}\label{coro:1}

Let $s_0$ be a fixed index and let $(j_s)_{s\ge s_0}$ be the sequence defined in \eqref{eq:j_s}.   There are constants $c_1, c_2, c_3$ and $c_4$ that depend only on $p,$ and for $u\ge c_1$ an event of probability at least $1-2 \exp \left(-c_2 u^2 2^{s_0}\right)$ on which the following holds. For every $1\le k\le p$, every $f^{(k)} \in \F^{(k)}$, and every $s\ge s_0$,
\begin{align*}
\bigg(\sum_{i<j_{s}}\left(\left(\Delta_{s} f^{(k)}\right)_i^*\right)^2\bigg)^{1 / 2}
\le c_3 u 2^{s / 2} \|\Delta_s f^{(k)}\|_{(c_4 u^2 2^s)}, & \quad  \bigg(\sum_{i\ge j_{s}}\left(\left(\Delta_{s} f^{(k)}\right)_i^*\right)^4\bigg)^{1 / 4} 
\le c_3 N^{1 / 4} \|\Delta_s f^{(k)}\|_{L_{8}},\\
\bigg(\sum_{i<j_{s}}\left(\left(\pi_{s} f^{(k)}\right)_i^*\right)^2\bigg)^{1 / 2}\le c_3 u 2^{s / 2}\|\pi_s f^{(k)}\|_{(c_4 u^2 2^s)},& \quad
\bigg(\sum_{i\ge j_{s}}\left(\left(\pi_{s} f^{(k)}\right)_i^*\right)^2\bigg)^{1 / 2}\le c_3 N^{1/2}\|\pi_{s}f^{(k)}\|_{L_4},
\\
\Bigg( \sum_{i\ge j_s}  \bigg(\bigg( \prod_{k'<k}\pi_{s+1} f^{(k')}\prod_{k'>k}\pi_{s} f^{(k')} \bigg)^*_{i}\bigg)^4\Bigg)^{1/4} &  \le c_3 N^{1/4} \bigg\|\prod_{k'<k}\pi_{s+1} f^{(k')}\prod_{k'>k}\pi_{s} f^{(k')}\bigg\|_{L_8},
\\
\Bigg( \sum_{i\ge j_s}  \bigg(\bigg( \prod_{k=1}^{p}\pi_{s+1} f^{(k)} \bigg)^*_{i}\bigg)^2\Bigg)^{1/2} &  \le c_3 N^{1/2} \bigg\|\prod_{k=1}^{p}\pi_{s+1} f^{(k)}\bigg\|_{L_4}.
\end{align*}

\end{corollary}

\begin{proof}[Proof of Corollary \ref{coro:1}]
 By the definition of an admissible sequence, we have $|\F^{(k)}_s| \le 2^{2^s}$ for every $1\le k\le p$. Furthermore, $|\F^{(k)}_s|\cdot |\F^{(k)}_{s+1}| \le 2^{2^s} \cdot 2^{2^{s+1}}\le 2^{2^{s+2}},$ which implies that for $\Delta_s f^{(k)} = \pi_{s+1}f^{(k)}-\pi_s f^{(k)}$, the cardinality of the class satisfies $\big|\{\Delta_s f^{(k)}: f^{(k)} \in \F^{(k)} \}\big| \le 2^{2^{s+2}}< 2^{p2^{s+2}}$. Therefore, the first two inequalities follow by applying Lemma~\ref{lemma:mendelson} to the class $\{\Delta_s f^{(k)}:f^{(k)}\in \F^{(k)}\}$ for all $1\le k\le p$ and $s\ge s_0$, with the choice $r=4$. 
 
 Similarly, the third and fourth inequalities follow by applying Lemma~\ref{lemma:mendelson} to the class $\F^{(k)}_s$ for all $1\le k\le p, s\ge s_0$, with $r=2$. The fifth inequality follows by applying Lemma~\ref{lemma:mendelson} to the class $\left\{\prod_{k'<k}\pi_{s+1} f^{(k')}\prod_{k'>k}\pi_{s} f^{(k')}:f^{(k')}\in \F^{(k')}\right\}$ for all $1\le k\le p, s\ge s_0$, with $r=4$, and using the bound
 \[
 \bigg|\bigg\{\prod_{k'<k}\pi_{s+1} f^{(k')}\prod_{k'>k}\pi_{s} f^{(k')}:f^{(k')}\in \F^{(k')}\bigg\}\bigg|\le \prod_{k'<k} 2^{2^{s+1}} \prod_{k'>k} 2^{2^s}\le 2^{(p-1) 2^{s+1}}<2^{p2^{s+2}}.
 \]
 The last inequality is obtained by applying Lemma~\ref{lemma:mendelson} to the class $\left\{\prod_{k=1}^{p}\pi_{s+1}f^{(k)}:f^{(k)}\in \F^{(k)}\right\}$ for all $s\ge s_0$, with $r=2$, and using the fact that
 \[
\bigg| \bigg\{\prod_{k=1}^{p} \pi_{s+1} f^{(k)}:f^{(k)}\in \F^{(k)} \bigg\} \bigg| \le \prod_{k=1}^{p} 2^{2^{s+1}}=2^{p 2^{s+1}}<2^{p 2^{s+2}}.
 \]
 
 Finally, applying a union bound over all $1\le k\le p, s\ge s_0$, we obtain that all the above inequalities simultaneously hold with probability at least
 \[
 1-4p\sum_{s\ge s_0 }2\exp(-c u^2 p 2^{s})\cdot 2^{p 2^{s+2}}\ge 1-2\exp(-c_2 u^2 2^{s_0}),
 \]
provided that $u\ge c_1$, where $c_1$ and $c_2$ are constants depending only on $p$.
\end{proof}

\subsection{Proof of Theorem \ref{thm:main2}}\label{subsec:main2}
This subsection contains the proof of our second main result, Theorem \ref{thm:main2}. The proof relies on the next result, Theorem \ref{thm:key_lemma}, which establishes a high-probability upper bound on the $\ell_{2}$-norm of the product of coordinate projection vectors of the function classes $(\F^{(k)})_{k=1}^{p}$. The proof of Theorem \ref{thm:key_lemma}, along with some discussion, is deferred to Subsection \ref{subsec:prooflemma}.

\begin{theorem}\label{thm:key_lemma}
Under the setting of Theorem \ref{thm:main2}, for any $u>0$, it holds with probability at least $1-\exp(-u^2)$ that
\begin{align*}
&\sup_{f^{(k)}\in \F^{(k)},1\le k\le p} \bigg(\sum_{i=1}^{N}\prod_{k=1}^{p} \left(f^{(k)}(X_i^{(k)})\right)^2\bigg)^{1/2}\\
& \qquad \lesssim_p  \bigg(\prod_{k=1}^{p} d_{\psi_2}(\F^{(k)})\bigg)  \bigg(\sqrt{N}+u\sum_{k=1}^{p} \bar{\gamma}(\F^{(k)},\psi_2)\prod_{k'\ne k}\Big(\bar{\gamma}(\F^{(k')},\psi_2)+(\log N)^{1/2}+u\Big)\bigg),
\end{align*}
where $\bar{\gamma}(\F^{(k)},\psi_2):=\gamma(\F^{(k)},\psi_2)/d_{\psi_2}(\F^{(k)})$ for each $1\le k\le p$.
\end{theorem}

\begin{remark}\label{remark:optimality}
Integrating the tail bound in Theorem~\ref{thm:key_lemma} yields
\begin{align}\label{eq:aux_key_lemma}
    &\E \sup_{f^{(k)}\in \F^{(k)},1\le k\le p} \bigg(\sum_{i=1}^{N}\prod_{k=1}^{p} \left(f^{(k)}(X_i^{(k)})\right)^2\bigg)^{1/2} \nonumber\\
    &\lesssim_p \bigg(\prod_{k=1}^{p} d_{\psi_2}(\F^{(k)})\bigg)  \bigg(\sqrt{N}+\sum_{k=1}^{p} \bar{\gamma}(\F^{(k)},\psi_2)\prod_{k'\ne k}\Big(\bar{\gamma}(\F^{(k')},\psi_2)+(\log N)^{1/2}\Big)\bigg) \nonumber \\
    &\asymp_p \bigg(\prod_{k=1}^{p} d_{\psi_2}(\F^{(k)})\bigg)  \bigg(\sqrt{N}+\prod_{k=1}^{p} \Big(\bar{\gamma}(\F^{(k)},\psi_2)+(\log N)^{1/2}\Big)\bigg).
\end{align}
The expectation bound \eqref{eq:aux_key_lemma} is optimal up to a constant depending only on $p$; in particular, the logarithmic factors are unavoidable. We demonstrate the sharpness of \eqref{eq:aux_key_lemma} with a simple example in Appendix~\ref{app:B}.
\end{remark}

Now we are ready to prove Theorem \ref{thm:main2}.

\begin{proof}[Proof of Theorem \ref{thm:main2}]

We first observe that the condition $0 \in \F^{(k)}$ or the symmetry of $\F^{(k)}$ (i.e., $f\in \F^{(k)}$ implies $-f \in \F^{(k)}$) ensures that $\gamma(\mcF^{(k)}, \psi_2)\gtrsim d_{\psi_2}(\F^{(k)})$ (see \cite[Lemma 4.6]{al2025sharp}), a fact that we will use repeatedly in the proof.

Let $(\mcF^{(k)}_s)_{s\ge 0}$ be an admissible sequence of $\mcF^{(k)}.$ Let $u\ge 4$ and $s_0\ge 0$. For any $f^{(k)}\in \F^{(k)}$ and $s\ge 0$, let $\pi_s f^{(k)}$ denote the nearest point in $\F^{(k)}_s$ to $f^{(k)}$ with respect to the $(u^2 2^s)$-norm under the probability measure $\mu^{(k)}$. 

The proof will use \emph{symmetrization}, an important technique in empirical process theory \cite[Chapter 2.3]{van2023weak}. Let $(\varepsilon_i)_{i=1}^{N}$ be independent symmetric Bernoulli random variables that are independent of $(X^{(k)}_i)_{i=1}^N, 1\le k\le p$. We seek to establish a high-probability upper bound on the supremum of the Bernoulli process given by
\[
\sup_{f^{(k)}\in \F^{(k)},1\le k\le p} \bigg|\sum_{i=1}^{N}\varepsilon_i \prod_{k=1}^{p} f^{(k)} (X^{(k)}_i) \bigg|.
\]
To that end, we first note that any $\prod_{k=1}^{p} f^{(k)}$ can be represented as a chain, or telescoping sum, of links of the form $\prod_{k=1}^{p} \pi_{s+1} f^{(k)}-\prod_{k=1}^{p} \pi_{s} f^{(k)}.$ Namely 
\begin{align*}
&\prod_{k=1}^{p} f^{(k)} = \sum_{s\ge s_0} \bigg(\prod_{k=1}^{p} \pi_{s+1} f^{(k)}-\prod_{k=1}^{p} \pi_{s} f^{(k)}\bigg) +\prod_{k=1}^{p} \pi_{s_0} f^{(k)}\\
&=\sum_{s\ge s_0}\bigg(\sum_{k=1}^{p} (\pi_{s+1} f^{(1)})\cdots (\pi_{s+1} f^{(k-1)}) (\pi_{s+1} f^{(k)}-\pi_{s} f^{(k)})(\pi_{s} f^{(k+1)})\cdots (\pi_{s} f^{(p)})\bigg)+ \prod_{k=1}^{p} \pi_{s_0} f^{(k)},
\end{align*}
provided that $\pi^{(k)}_{s} f^{(k)}\to f^{(k)}$ as $s \to \infty$. Here, in the second line, for each $s$, we used the telescoping identity $a_1 \cdots a_{p}-b_1 \cdots b_{p}=\sum_{k=1}^{p} a_1 \cdots a_{k-1}\left(a_k-b_k\right) b_{k+1} \cdots b_{p}$ with $a_{k}=\pi_{s+1}f^{(k)}$ and $b_{k}=\pi_{s}f^{(k)}$.

To ease notation, we write $(\pi_{s}f^{(k)})_i:=(\pi_{s}f^{(k)})(X_i^{(k)})$, $\Delta_s f^{(k)}:=\pi_{s+1} f^{(k)}-\pi_s f^{(k)}$ and $(\Delta_s f^{(k)})_i := (\Delta_s f^{(k)})(X^{(k)}_i)$. We will use this notation throughout the remainder of the paper. By the triangle inequality, 
\begin{align}\label{eq:main2_aux1}
&\bigg|\sum_{i=1}^N \varepsilon_i \prod_{k=1}^{p} f^{(k)} (X^{(k)}_i)\bigg|\nonumber\\
&=\bigg| \sum_{i=1}^{N} \varepsilon_i \bigg[ \sum_{s\ge s_0}\bigg(\sum_{k=1}^{p} (\pi_{s+1} f^{(1)})_i\cdots (\pi_{s+1} f^{(k-1)})_i (\Delta_s f^{(k)})_i(\pi_{s} f^{(k+1)})_i\cdots (\pi_{s} f^{(p)})_i\bigg)+ \prod_{k=1}^{p} (\pi_{s_0} f^{(k)})_i \bigg]\bigg| \nonumber\\
&\le \sum_{s\ge s_0} \sum_{k=1}^{p} \bigg|\sum_{i=1}^N \varepsilon_i (\pi_{s+1} f^{(1)})_i\cdots (\pi_{s+1} f^{(k-1)})_i (\Delta_s f^{(k)})_i(\pi_{s} f^{(k+1)})_i\cdots (\pi_{s} f^{(p)})_i\bigg|+\bigg|\sum_{i=1}^N \varepsilon_i \prod_{k=1}^{p} (\pi_{s_0} f^{(k)})_i \bigg| \nonumber\\
&=\sum_{s\ge s_0}\sum_{k=1}^{p}\bigg| \sum_{i=1}^{N} \varepsilon_i \bigg((\Delta_{s}f^{(k)})_i\prod_{k'<k}(\pi_{s+1} f^{(k')})_{i}\prod_{k'>k}(\pi_{s} f^{(k')})_i\bigg) \bigg|+\bigg|\sum_{i=1}^N \varepsilon_i \prod_{k=1}^{p} (\pi_{s_0} f^{(k)})_i \bigg|.
\end{align}

Let $(j_{s})_{s\ge s_0}$ be the sequence defined in \eqref{eq:j_s}. For each $s\ge s_0$, let $I_s$ be the union of the $j_{s}-1$ largest coordinates of $(|\Delta_{s} f^{(k)}|_i)_{i=1}^N$ and the $j_{s}-1$ largest coordinates of 
\[
\bigg(\Big|\prod_{k'<k}(\pi_{s+1} f^{(k')})_{i}\prod_{k'>k}(\pi_{s} f^{(k')})_i\Big|\bigg)_{i=1}^{N},
\]
for every $1\le k\le p$. Then $|I_{s}|\le 2p(j_{s}-1)$. 
We apply Lemma \ref{lemma:Hoffding} with $t=u p^{1/2}2^{s/2}$ and the set $I_s$ to the summand indexed by $s$ and $k$ in \eqref{eq:main2_aux1}. This implies that, for every $s, k$ and every $f^{(k)}\in \F^{(k)} (1\le k\le p)$, with probability at least $1-2 \exp (-u^2 p 2^{s-1})$, the following holds:
\begin{align*}
 &\bigg| \sum_{i=1}^{N} \varepsilon_i \bigg((\Delta_{s}f^{(k)})_i\prod_{k'<k}(\pi_{s+1} f^{(k')})_{i}\prod_{k'>k}(\pi_{s} f^{(k')})_i\bigg) \bigg|\\
     &\le \sum_{i\in I_s}\bigg|(\Delta_{s}f^{(k)})_i\prod_{k'<k}(\pi_{s+1} f^{(k')})_{i}\prod_{k'>k}(\pi_{s} f^{(k')})_i\bigg|\\
     &\quad +u p^{1/2} 2^{s/2} \bigg(\sum_{i\in I_s^c} (\Delta_{s}f^{(k)})^2_i\prod_{k'<k}(\pi_{s+1} f^{(k')})^2_{i}\prod_{k'>k}(\pi_{s} f^{(k')})^2_i \bigg)^{1/2}\\
     &\overset{\text{($\star$)}}{\le} \bigg(\sum_{i\in I_s} (\Delta_s f^{(k)})^2_i\bigg)^{1/2} \bigg(\sum_{i\in I_s} \prod_{k'<k}(\pi_{s+1} f^{(k')})^2_{i}\prod_{k'>k}(\pi_{s} f^{(k')})^2_i\bigg)^{1/2}\\
     &\quad +u p^{1/2} 2^{s/2}\bigg(\sum_{i\in I^c_s} (\Delta_s f^{(k)})^4_i\bigg)^{1/4} \bigg(\sum_{i\in I^c_s} \prod_{k'<k}(\pi_{s+1} f^{(k')})^4_{i}\prod_{k'>k}(\pi_{s} f^{(k')})^4_i\bigg)^{1/4}\\
&\overset{\text{($\star \star$)}}{\lesssim}_p \bigg(\sum_{i<j_{s}}\left(\left(\Delta_{s} f^{(k)}\right)_i^*\right)^2\bigg)^{1 / 2} \sup_{f^{(k')}\in \F^{(k')},k'\ne k}\bigg(\sum_{i=1}^{N} \prod_{k'\ne k} (f^{(k')})^2_{i}\bigg)^{1/2}\\
&\quad+u 2^{s/2}\bigg(\sum_{i\ge j_{s}}\left(\left(\Delta_{s} f^{(k)}\right)_i^*\right)^4\bigg)^{1 / 4}\Bigg( \sum_{i\ge j_s}  \bigg(\bigg( \prod_{k'<k}\pi_{s+1} f^{(k')}\prod_{k'>k}\pi_{s} f^{(k')} \bigg)^*_{i}\bigg)^4\Bigg)^{1/4},
\end{align*}
where ($\star$) follows by Cauchy-Schwarz inequality and in ($\star\star$) we use that $ \sum_{i\in I_s} (\Delta_s f^{(k)})^2_i\le 2p \sum_{i<j_s} \big(\big(\Delta_s f^{(k)}\big)^*_i \big)^2 $ since $|I_{s}|\le 2p(j_s-1)$, and that
\begin{align*}
 \bigg(\sum_{i\in I_s} \prod_{k'<k}(\pi_{s+1} f^{(k')})^2_{i}\prod_{k'>k}(\pi_{s} f^{(k')})^2_i\bigg)^{1/2}\le \sup_{f^{(k')}\in \F^{(k')},k'\ne k}\bigg(\sum_{i=1}^{N} \prod_{k'\ne k} (f^{(k')})^2_{i}\bigg)^{1/2}.
\end{align*}

Further, repeating this argument for the last term in \eqref{eq:main2_aux1} gives that, for every $f^{(k)}\in \F^{(k)} (1\le k\le p)$, with probability at least $1-2\exp(- u^2 p 2^{s_0-1})$,
\begin{align*}
\bigg|\sum_{i=1}^{N} \varepsilon_i \prod_{k=1}^{p}(\pi_{s_0} f^{(k)})_i \bigg|&\lesssim_p \bigg(\sum_{i<j_{s_0}}\left(\left(\pi_{s_0} f^{(1)}\right)_i^*\right)^2\bigg)^{1 / 2} \sup_{f^{(k')}\in \F^{(k')},k'\ne 1}\bigg(\sum_{i=1}^{N} \prod_{k'\ne 1} (f^{(k')})^2_{i}\bigg)^{1/2}\\
&\quad +u 2^{s_0/2} \bigg(\sum_{i\ge j_{s_0}}\left(\left(\pi_{s_0} f^{(1)}\right)_i^*\right)^4\bigg)^{1 / 4}\Bigg( \sum_{i\ge j_{s_0}}  \bigg(\bigg( \prod_{k'\ne 1}\pi_{s_0} f^{(k')} \bigg)^*_{i}\bigg)^4\Bigg)^{1/4}.
\end{align*}

Observe that $|\{\Delta_s f^{(k)}: f^{(k)}\in \F^{(k)}\}|\le 2^{2^{s}}\cdot 2^{2^{s+1}}\le 2^{2^{s+2}}$ and
\[
\bigg|\bigg\{\prod_{k'<k}\pi_{s+1} f^{(k')}\prod_{k'>k}\pi_{s} f^{(k')}: f^{(k')}\in \F^{(k')} \bigg\} \bigg|\le \prod_{k'<k}2^{2^{s+1}} \prod_{k'>k}2^{2^{s}}\le 2^{(p-1)2^{s+1}},
\]
\[
\bigg|\bigg\{\prod_{k=1}^{p}\pi_{s_0}f^{(k)}: f^{(k)}\in \F^{(k)} \bigg\} \bigg|\le \prod_{k=1}^{p}2^{2^{s_0}}= 2^{p2^{s_0}}.
\]
For $u\ge 4$, we sum the probabilities for all $s\ge s_0$ and $1\le k\le p$ to obtain that, with $(\varepsilon_i)_{i=1}^{N}$ probability at least 
\[
1-2\exp(-u^2 p 2^{s_0-1})\cdot 2^{p2^{s_0}} -\sum_{s\ge s_0} 2\exp(-u^2p 2^{s-1})\cdot 2^{(p-1)2^{s+1}}\cdot 2^{2^{s+2}}\ge 1-2\exp(-u^2 2^{s_0}/8),
\]
for every $f^{(k)}\in \F^{(k)} (1\le k\le p)$,
\begin{align*}
    &\bigg|\sum_{i=1}^N \varepsilon_i \prod_{k=1}^{p} f^{(k)} (X^{(k)}_i)\bigg|\\
    &\le \sum_{s\ge s_0}\sum_{k=1}^{p}\bigg| \sum_{i=1}^{N} \varepsilon_i \bigg((\Delta_{s}f^{(k)})_i\prod_{k'<k}(\pi_{s+1} f^{(k')})_{i}\prod_{k'>k}(\pi_{s} f^{(k')})_i\bigg) \bigg|+\bigg|\sum_{i=1}^N \varepsilon_i \prod_{k=1}^{p} (\pi_{s_0} f^{(k)})_i \bigg|\\
    &\lesssim_p  \sum_{k=1}^{p} \Bigg(\sum_{s\ge s_0}\bigg(\sum_{i<j_{s}}\left(\left(\Delta_{s} f^{(k)}\right)_i^*\right)^2\bigg)^{1 / 2}+\bigg(\sum_{i<j_{s_0}}\left(\left(\pi_{s_0} f^{(k)}\right)_i^*\right)^2\bigg)^{1 / 2}\Bigg) \sup_{f^{(k')}\in \F^{(k')},k'\ne k}\bigg(\sum_{i=1}^{N} \prod_{k'\ne k} (f^{(k')})^2_{i}\bigg)^{1/2}\\
&\quad+u \sum_{k=1}^{p}\sum_{s\ge s_0}2^{s/2}\bigg(\sum_{i\ge j_{s}}\left(\left(\Delta_{s} f^{(k)}\right)_i^*\right)^4\bigg)^{1 / 4}\Bigg( \sum_{i\ge j_s}  \bigg(\bigg( \prod_{k'<k}\pi_{s+1} f^{(k')}\prod_{k'>k}\pi_{s} f^{(k')} \bigg)^*_{i}\bigg)^4\Bigg)^{1/4}\\
&\quad + u 2^{s_0/2} \bigg(\sum_{i\ge j_{s_0}}\left(\left(\pi_{s_0} f^{(1)}\right)_i^*\right)^4\bigg)^{1 / 4}\Bigg( \sum_{i\ge j_{s_0}}  \bigg(\bigg( \prod_{k'\ne 1}\pi_{s_0} f^{(k')} \bigg)^*_{i}\bigg)^4\Bigg)^{1/4}.
\end{align*}

On the other hand, we know from Corollary \ref{coro:1} that with $(X^{(k)}_i)_{i=1}^{N} (1\le k\le p)$ probability at least $1-2\exp(-c_1u^2 2^{s_0})$, the four inequalities in Corollary \ref{coro:1} hold for every $1\le k\le p$, every $f^{(k)} \in \F^{(k)}$ and $s\ge s_0$ simultaneously. Therefore, combining the $(\varepsilon_i)_{i=1}^N$ probability and the $(X^{(k)}_i)_{i=1}^N (1\le k\le p)$ probability implies that there exists a constant $\tilde{c}(p)$ depending only on $p$ such that, for $u\ge c(p)\lor 4$, it holds with probability at least
\[
1-2\exp(-u^2 2^{s_0}/8)-2\exp(-c_1u^2 2^{s_0})\ge 1-2 \exp \left(-\tilde{c}(p) u^2 2^{s_{0}}\right)
\]
that, for every $1\le k\le p$ and every $f^{(k)}\in \F^{(k)}$,
\begin{align}\label{eq:main2_aux_special}
    &\bigg|\sum_{i=1}^N \varepsilon_i \prod_{k=1}^{p} f^{(k)} (X^{(k)}_i)\bigg| \nonumber \\
    &\lesssim_p \sum_{k=1}^{p}  u \bigg(\sum_{s\ge s_0}2^{s/2}\|\Delta_s f^{(k)}\|_{(c_3 u^2 2^s)}+2^{s_0/2}\|\pi_{s_0}f^{(k)}\|_{(c_3 u^2 2^{s_0})}\bigg)\sup_{f^{(k')}\in \F^{(k')},k'\ne k}\bigg(\sum_{i=1}^{N} \prod_{k'\ne k} (f^{(k')})^2_{i}\bigg)^{1/2} \nonumber\\
    &\quad +u \sum_{k=1}^{p} \bigg(\sum_{s\ge s_0} 2^{s/2} N^{1/4} \|\Delta_s f^{(k)}\|_{L_8} N^{1/4}\bigg\|\prod_{k'<k}\pi_{s+1} f^{(k')}\prod_{k'>k}\pi_{s} f^{(k')}\bigg\|_{L_8}\bigg)\nonumber\\
    &\quad+ u2^{s_0/2} N^{1/4} \|\pi_{s_0}f^{(1)}\|_{L_8} N^{1/4} \bigg\|\prod_{k'\ne 1}\pi_{s_0} f^{(k')}\bigg\|_{L_8} \\
    &\overset{\text{($\star$)}}{\lesssim}_p u\sum_{k=1}^{p} \widetilde{\Lambda}_{s_0,cu}(\F^{(k)}) \sup_{f^{(k')}\in \F^{(k')},k'\ne k}\bigg(\sum_{i=1}^{N} \prod_{k'\ne k} (f^{(k')})^2_{i}\bigg)^{1/2} \nonumber \\
    &\quad + uN^{1/2}\sum_{k=1}^{p}\bigg(\prod_{k'\ne k} d_{\psi_2}(\F^{(k')})\bigg) \bigg(\sum_{s\ge s_0} 2^{s/2} \|\Delta_s f^{(k)}\|_{\psi_2}\bigg)+u2^{s_0/2} N^{1/2} \prod_{k=1}^{p} d_{\psi_2}(\F^{(k)}) \nonumber\\
    &\overset{\text{($\star\star$)}}{\lesssim}_p u\sum_{k=1}^{p} \widetilde{\Lambda}_{s_0,cu}(\F^{(k)}) \sup_{f^{(k')}\in \F^{(k')},k'\ne k}\bigg(\sum_{i=1}^{N} \prod_{k'\ne k} (f^{(k')})^2_{i}\bigg)^{1/2} \nonumber\\
    &\quad + uN^{1/2}\sum_{k=1}^{p }\bigg(\prod_{k'\ne k} d_{\psi_2}(\F^{(k')})\bigg) \bigg(\gamma(\F^{(k)},\psi_2)+2^{s_0/2}d_{\psi_2}(\F^{(k)})\bigg),\nonumber
\end{align}
where ($\star$) follows by choosing almost optimal admissible sequences in $\F^{(k)}$ for $1\le k\le p$ and using the definition of $\widetilde{\Lambda}_{s_0,u}(\F^{(k)})$, together with $\|\Delta_s f^{(k)}\|_{L_8}\lesssim \|\Delta_s f^{(k)}\|_{\psi_2}$ and H\"older's inequality. In particular, for each $k$ and $s$,
\[
\bigg\|\prod_{k'<k}\pi_{s+1} f^{(k')}\prod_{k'>k}\pi_{s} f^{(k')}\bigg\|_{L_8}\le \prod_{k'<k}  \|\pi_{s+1} f^{(k')}\|_{L_{8(p-1)}}  \prod_{k'>k} \|\pi_{s}f^{(k')}\|_{L_{8(p-1)}} \lesssim_p \prod_{k'\ne k} d_{\psi_2}(\F^{(k')}),
\]
and
\[
\|\pi_{s_0}f^{(1)}\|_{L_8} \bigg\|\prod_{k'\ne 1} \pi_{s_0}f^{(k')}\bigg\|_{L_8}\le \|\pi_{s_0}f^{(1)}\|_{L_8} \prod_{k'\ne 1} \|\pi_{s_0}f^{(k')}\|_{L_{8(p-1)}}\lesssim_p \prod_{k=1}^{p} d_{\psi_2}(\F^{(k)}),
\]
where we used Hölder's inequality in the form $   \big\|\prod_{j=1}^{p-1}Z_j\big\|_{L_8}\le \prod_{j=1}^{p-1} \|Z_j\|_{L_{8(p-1)}}$ for any random variables $Z_1,\ldots,Z_{p-1}$; and ($\star\star$) follows by noting that
\begin{align*}
\sum_{s\ge s_0} 2^{s/2}\|\Delta_s f^{(k)}\|_{\psi_2}&\le \sum_{s\ge s_0} 2^{s/2}(\|f-\pi_{s}f^{(k)}\|_{\psi_2}+\|f-\pi_{s+1}f^{(k)}\|_{\psi_2}) \lesssim \gamma(\F^{(k)},\psi_2).
\end{align*}

Now we apply Theorem~\ref{thm:key_lemma} and a union bound to obtain the high-probability bound: for any $u>0$, it holds with probability at least $1-p\exp(-u^2)$ that, for every $1\le k\le p$,
\begin{align*}
 &\sup_{f^{(k')}\in \F^{(k')},k'\ne k}\bigg(\sum_{i=1}^{N} \prod_{k'\ne k} (f^{(k')})^2_{i}\bigg)^{1/2}\\
 &\qquad \lesssim_p  \bigg(\prod_{k'\ne k} d_{\psi_2}(\F^{(k')})\bigg)  \bigg(\sqrt{N}+u\sum_{k'\ne k} \bar{\gamma}(\F^{(k')},\psi_2)\prod_{k''\ne k,k'}\Big(\bar{\gamma}(\F^{(k'')},\psi_2)+(\log N)^{1/2}+u\Big)\bigg).
\end{align*}
Moreover, by Lemma \ref{lemma:Lambda_Gamma}, $ \Lambda_{s_0, cu}(\F^{(k)}) \lesssim \gamma(\F^{(k)}, \psi_2)$, and hence 
\[
\widetilde{\Lambda}_{s_0,cu}(\F^{(k)})=\Lambda_{s_0, cu}(\F^{(k)})+2^{s_0 / 2} \sup_{f^{(k)} \in \F^{(k)}}\|\pi_{s_0} f^{(k)}\|_{(c^2 u^2 2^{s_0})}\lesssim \gamma(\F^{(k)},\psi_2)+2^{s_0/2}d_{\psi_2}(\F^{(k)}).
\]
Consequently, by taking $s_0\asymp 1$, we have shown that there exist constants $c_1(p), c_2(p)$ and $c_3(p)$ such that, for any $u\ge c_1(p)$, it holds with probability at least $1-2\exp(-\tilde{c}(p)u^2 2^{s_0})-p\exp(-u^2)\ge 1-c_2(p)\exp(-c_3(p)u^2)$ that

\begin{align*}
 &\sup_{f^{(k)}\in \F^{(k)},1\le k\le p}   \bigg|\sum_{i=1}^N \varepsilon_i \prod_{k=1}^{p} f^{(k)} (X^{(k)}_i)\bigg|\\
 &\lesssim_p u\sum_{k=1}^{p} \widetilde{\Lambda}_{s_0,cu}(\F^{(k)}) \sup_{f^{(k')}\in \F^{(k')},k'\ne k}\bigg(\sum_{i=1}^{N} \prod_{k'\ne k} (f^{(k')})^2_{i}\bigg)^{1/2}\\
 &\quad + uN^{1/2}\sum_{k=1}^{p }\bigg(\prod_{k'\ne k} d_{\psi_2}(\F^{(k')})\bigg) \bigg(\gamma(\F^{(k)},\psi_2)+2^{s_0/2}d_{\psi_2}(\F^{(k)})\bigg)\\
 &\lesssim_p u\sum_{k=1}^{p} \gamma(\F^{(k)},\psi_2)  \bigg(\prod_{k'\ne k} d_{\psi_2}(\F^{(k')})\bigg)  \bigg(\sqrt{N}+u\sum_{k'\ne k} \bar{\gamma}(\F^{(k')},\psi_2)\prod_{k''\ne k,k'}\Big(\bar{\gamma}(\F^{(k'')},\psi_2)+(\log N)^{1/2}+u\Big)\bigg)\\
 &\quad + uN^{1/2}\sum_{k=1}^{p }\bigg(\prod_{k'\ne k} d_{\psi_2}(\F^{(k')})\bigg) \gamma(\F^{(k)},\psi_2)\\
 &= \bigg(\prod_{k=1}^{p}d_{\psi_2}(\F^{(k)})\bigg)u\sqrt{N}\bigg(\sum_{k=1}^{p}\bar{\gamma}(\F^{(k)},\psi_2)\bigg)\\
 &\quad +\bigg(\prod_{k=1}^{p}d_{\psi_2}(\F^{(k)})\bigg) u^2 \sum_{k\ne k'} \bar{\gamma}(\F^{(k)},\psi_2)\bar{\gamma}(\F^{(k')},\psi_2)\prod_{k''\ne k,k'}\Big(\bar{\gamma}(\F^{(k'')},\psi_2)+(\log N)^{1/2}+u\Big),
\end{align*}
where $\bar{\gamma}(\F^{(k)},\psi_2):=\gamma(\F^{(k)},\psi_2)/d_{\psi_2}(\F^{(k)})$ for $1\le k\le p$. 

Integrating the above tail bound (using Lemma~\ref{lemma:tail_to_moment} with $F(u)$ a polynomial with nonnegative coefficients; cf.~Example~\ref{ex:growth_positive_poly}) yields the following $q$-th moment estimate: for any $q\ge 1$,
\begin{align*}
&\left(\E \sup_{f^{(k)}\in \F^{(k)},1\le k\le p} \bigg|\sum_{i=1}^{N}\varepsilon_i \prod_{k=1}^{p} f^{(k)} (X^{(k)}_i) \bigg|^q\right)^{1/q} \\
&\lesssim_{p,q} \bigg(\prod_{k=1}^{p} d_{\psi_2}(\F^{(k)})\bigg)\Bigg(\sqrt{N}\bigg(\sum_{k=1}^{p}\bar{\gamma}(\F^{(k)},\psi_2)\bigg)\\
&\quad +\sum_{k\ne k'} \bar{\gamma}(\F^{(k)},\psi_2)\bar{\gamma}(\F^{(k')},\psi_2)\prod_{k''\ne k,k'}\Big(\bar{\gamma}(\F^{(k'')},\psi_2)+(\log N)^{1/2}\Big)\Bigg) \\
&\lesssim_{p,q} \bigg(\prod_{k=1}^{p} d_{\psi_2}(\F^{(k)})\bigg)\left(\sqrt{N}\bigg(\sum_{k=1}^{p}\bar{\gamma}(\F^{(k)},\psi_2)\bigg)+\prod_{k=1}^{p}\Big(\bar{\gamma}(\F^{(k)},\psi_2)+(\log N)^{1/2}\Big)\right).
\end{align*}

Finally, by the symmetrization inequality for empirical processes \cite[Lemma 2.3.1]{van2023weak}, we have
\begin{align*}
&\left(\E \sup_{f^{(k)} \in \mathcal{F}^{(k)},1\le k\le p}\bigg|\frac{1}{N} \sum_{i=1}^N \prod_{k=1}^p f^{(k)} (X^{(k)}_i)-\mathbb{E} \prod_{k=1}^p f^{(k)} (X^{(k)})\bigg|^q\right)^{1/q} \\
&\le 2 \left(\E \sup_{f^{(k)}\in \F^{(k)},1\le k\le p} \bigg|\frac{1}{N}\sum_{i=1}^{N}\varepsilon_i \prod_{k=1}^{p} f^{(k)} (X^{(k)}_i) \bigg|^q\right)^{1/q} \\
&\lesssim_{p,q} \bigg(\prod_{k=1}^{p} d_{\psi_2}(\F^{(k)})\bigg)\bigg(\frac{\sum_{k=1}^{p}\bar{\gamma}(\F^{(k)},\psi_2)}{\sqrt{N}}+\frac{\prod_{k=1}^{p}\big(\bar{\gamma}(\F^{(k)},\psi_2)+(\log N)^{1/2}\big)}{N}\bigg),
\end{align*}
which concludes the proof.

\end{proof}

\subsection{Proof of Theorem \ref{thm:key_lemma}}\label{subsec:prooflemma}

This subsection presents the proof of Theorem \ref{thm:key_lemma}, which establishes a high-probability upper bound on the $\ell_{2}$-norm of the product of coordinate projection vectors associated with the function classes $(\F^{(k)})_{k=1}^{p}$:
\[
\sup_{f^{(k)}\in \F^{(k)},1\le k\le p} \bigg(\sum_{i=1}^{N}\prod_{k=1}^{p} \left(f^{(k)}(X_i^{(k)})\right)^2\bigg)^{1/2}.
\]
The proof is based on generic chaining for the classes $(\F^{(k)})_{k=1}^{p}$, and builds on ideas of Mendelson \cite{mendelson2016upper} for controlling multiplier and product processes, including the use of order statistics to separate the contributions of the largest and smallest coordinates of the increments (via monotone sums). In this sense, our result can be viewed as a multi-product analogue of \cite[Lemma 2.3]{mendelson2016upper}, which corresponds (roughly) to the case $p=1$. The main new difficulty is that we must handle $p$ potentially different function classes evaluated at $p$ different random variables, and leverage the resulting multilinear structure to carry out chaining across these classes. 

A crucial step in controlling the increments is to incorporate $\ell_{\infty}$-norm bounds on the relevant coordinate projection vectors at appropriate stages of the chaining argument, which introduces $\sqrt{\log N}$ factors. Although this bound may appear crude, it is in fact sharp: these $\sqrt{\log N}$ terms are unavoidable in general; see Appendix~\ref{app:B}. In particular, even if some class $\F^{(k)}$ is ``simple'' in a global complexity sense, its contribution to the overall bound still necessarily includes at least a $\sqrt{\log N}$ term, reflecting its largest coordinate fluctuation in $\ell_{\infty}$. This viewpoint provides a principled way to understand the behavior of products of coordinate projection vectors across heterogeneous classes $(\F^{(k)})_{k=1}^{p}$ and different underlying random variables. 

We note that when $ \F^{(1)}= \cdots=\F^{(p)}=\F$, $X^{(1)}=\cdots =X^{(p)}=X$, and $X^{(1)}_i=\cdots = X^{(p)}_i=X_i \ $ for all $1\le i\le N,$ Theorem \ref{thm:key_lemma} reduces to \cite[Theorem 4.5]{al2025sharp}. However, the proof presented here is genuinely different from that of \cite[Theorem 4.5]{al2025sharp}, which relies on chaining with a Young-function argument and separates the sub-Gaussian and sub-exponential regimes. In contrast, our setting involves heterogeneous function classes evaluated at different random variables and requires incorporating $\ell_{\infty}$-norm control within the chaining procedure, together with a sharp accounting of the resulting $\sqrt{\log N}$ factors. These features ---and in particular the necessity of the $\sqrt{\log N}$ terms--- are not captured by the approach in \cite{al2025sharp}.

We begin by establishing several lemmas needed for the proof of Theorem \ref{thm:key_lemma}.

The following result provides a sharp high-probability upper bound for the $\ell_{\infty}$-norm of the coordinate projection vector $(f(X_i))_{i=1}^{N}$.

\begin{lemma}\label{lemma:infinity_norm_bound}
  For any $u>0$, it holds with probability at least $1-\exp(-u^2)$ that
  \[
\sup_{f\in \F} \max_{1\le i\le N} |f(X_i)|\lesssim \gamma(\F,\psi_2)+d_{\psi_2}(\F) (\sqrt{\log N}+u).
  \]
\end{lemma}


\begin{proof}[Proof of Lemma \ref{lemma:infinity_norm_bound}]
    By \cite[Theorem 3.2]{dirksen2015tail} (see also \cite[Theorem 4.5]{al2025sharp}), for any $t\ge 1$, it holds with probability at least $1-\exp(-t^2)$ that
    \[
    \sup_{f\in \F} |f(X)|\lesssim \gamma(\F,\psi_2)+d_{\psi_2}(\F)t.
    \]
   Taking a union bound yields that, with probability at least $1-N\exp(-t^2),$
    \[
    \max_{1\le i\le N} \sup_{f\in \F} |f(X_i)|\lesssim \gamma(\F,\psi_2)+d_{\psi_2}(\F)t.
    \]
    Setting $t=\sqrt{\log N}+u$, we have with probability at least $1-N\exp(-(\sqrt{\log N}+u)^2)\ge 1-\exp(-u^2)$ that
    \[
    \max_{1\le i\le N} \sup_{f\in \F} |f(X_i)|\lesssim \gamma(\F,\psi_2)+d_{\psi_2}(\F)(\sqrt{\log N}+u).
    \]
    The proof is complete, since $\sup_{f\in \F} \max_{1\le i\le N} |f(X_i)|=\max_{1\le i\le N} \sup_{f\in \F}|f(X_i)|$.
\end{proof}


The next lemma is a consequence of \cite[Lemma 3.2]{mendelson2016upper}, as noted in \cite[Remark 3.3]{mendelson2016upper}. We state the result in the special case obtained by setting $q=4, r=2, t=2, \rho=3/2, \beta=1/2,$ and $\alpha=3/4$ in their notation. 


\begin{lemma}[{\cite[Lemma 3.2, Remark 3.3]{mendelson2016upper}}]\label{lemma:order_stats}
There exists an absolute constant $c_0>0$ such that the following holds. Let $Z\in L_{4}$ and let $Z_1,\ldots, Z_N$ be independent copies of $Z$. Fix $1\le w\le N$, and define
\begin{align}\label{eq:j_0}
\mathsf{j_{*}}=\left\lceil  \frac{c_0 w}{\log (4+eN/w)} \right\rceil.
\end{align}
Then, with probability at least $1-\exp(-w)$, the following inequality holds simultaneously for all $i\ge \mathsf{j_{*}}$:
\[
Z_i^*\le 2\|Z\|_{L_4} \left(\frac{eN}{i}\right)^{3/8},
\]
where $Z_i^*$ denotes the $i$-th largest value among $|Z_1|,\ldots, |Z_{N}|$.
\end{lemma}

We apply Lemma \ref{lemma:order_stats} to derive the following corollary.

\begin{corollary}\label{coro:order_stats}
There exist two constants $c_1$ and $c_2$ depending only on $p$ such that the following holds. For $u>c_1$ and $s_0>1$, let $(j_s)_{s\ge s_0}$ be defined as in \eqref{eq:j_s}. Then, with probability at least $1-2\exp(-c_2 u^2 2^{s_0})$, the following holds: for every $f^{(k)}\in \F^{(k)} (1\le k\le p)$ and every $s\ge s_0$,
\[
\Bigg(\sum_{i=j_{s-1}}^{j_{s}-1}\bigg(\bigg(\prod_{k=1}^{p}\pi_{s} f^{(k)}\bigg)_i^{*}\bigg)^{2}\Bigg)^{1 / 2}\le 2 \bigg\|\prod_{k=1}^{p}\pi_s f^{(k)}\bigg\|_{L_4}  \Bigg(\sum_{i=j_{s-1}}^{j_{s}-1}\left(\frac{e N}{i}\right)^{3/4}\Bigg)^{1 / 2}.
\]
\end{corollary}

\begin{proof}[Proof of Corollary \ref{coro:order_stats}]
For every $s\ge s_0$ and every $f^{(k)}\in \F^{(k)} \,(1\le k\le p)$, we apply Lemma \ref{lemma:order_stats} to the random variable $Z=\prod_{k=1}^{p} \pi_s f^{(k)}$ with $w=u^2 p 2^{s-1}$. Note that $\mathsf{j_{*}}$ coincides with $j_{s-1}$ defined in \eqref{eq:j_s}. We thus obtain that, with probability at least $1-\exp(-u^2 p 2^{s-1})$, for every $i\ge j_{s-1}$,
\[
\bigg(\prod_{k=1}^{p}\pi_s f^{(k)}\bigg)_i^* \le 2 \bigg\|\prod_{k=1}^{p}\pi_s f^{(k)}\bigg\|_{L_4} \left(\frac{eN}{i}\right)^{3/8},
\]
which implies that
\[
\Bigg(\sum_{i=j_{s-1}}^{j_{s}-1}\bigg(\bigg(\prod_{k=1}^{p}\pi_{s} f^{(k)}\bigg)_i^{*}\bigg)^{2}\Bigg)^{1 / 2}\le 2 \bigg\|\prod_{k=1}^{p}\pi_s f^{(k)}\bigg\|_{L_4}  \bigg(\sum_{i=j_{s-1}}^{j_{s}-1}\left(\frac{e N}{i}\right)^{3/4}\bigg)^{1 / 2}.
\]
Since
\[
\bigg|\bigg\{\prod_{k=1}^{p}\pi_{s}f^{(k)}: f^{(k)}\in \F^{(k)} \bigg\} \bigg|\le \prod_{k=1}^{p}2^{2^{s}}= 2^{p2^{s}},
\]
we take a union bound over all $s\ge s_0$ and all different functions $\prod_{k=1}^{p}\pi_{s}f^{(k)}$ to obtain that, for $u\ge c_1$, with probability at least $1-\sum_{s\ge s_0}2^{p2^s}\exp(-u^2 p 2^{s-1}) \ge 1-2\exp(-c_2 u^2 2^{s_0})$, for every $f^{(k)}\in \F^{(k)} \, (1\le k\le p)$ and every $s\ge s_0$:
\[
\Bigg(\sum_{i=j_{s-1}}^{j_{s}-1}\bigg(\bigg(\prod_{k=1}^{p}\pi_{s} f^{(k)}\bigg)_i^{*}\bigg)^{2}\Bigg)^{1 / 2}\le 2 \bigg\|\prod_{k=1}^{p}\pi_s f^{(k)}\bigg\|_{L_4}  \bigg(\sum_{i=j_{s-1}}^{j_{s}-1}\left(\frac{e N}{i}\right)^{3/4}\bigg)^{1 / 2}.
\]
This completes the proof of Corollary \ref{coro:order_stats}.
\end{proof}

We are now ready to prove Theorem \ref{thm:key_lemma}.

\begin{proof}[Proof of Theorem \ref{thm:key_lemma}]
We will use the same notation as in the proof of Theorem \ref{thm:main2}.
Let $(\mcF^{(k)}_s)_{s\ge 0}$ be an admissible sequence of $\mcF^{(k)}.$ Let $u\ge 4$ and $s_0\ge 0$. For any $f^{(k)}\in \F^{(k)}$ and $s\ge 0$, let $\pi_s f^{(k)}$ denote the nearest point in $\F^{(k)}_s$ to $f^{(k)}$ with respect to the $(u^2 2^s)$-norm under the probability measure $\mu^{(k)}$.  Let $(j_{s})_{s\ge s_0}$ be the non-decreasing sequence of integers defined in \eqref{eq:j_s}, and notice that $1\le j_{s}\le N+1$ for $s \ge s_0$.  Set $s_1>s_0$ to be the first integer for which $j_{s}=N+1$. To ease notation, we will write $\|f^{(k)}\|_{\ell^N_{\infty}}:= \max_{1\le i\le N}|f^{(k)}(X_i^{(k)})| = 
\max_{1\le i\le N}|(f^{(k)})_i|.$

 
Observe that $f^{(k)}=\sum_{s\ge s_1}\Delta_s f^{(k)}+\pi_{s_1}f^{(k)}$. We repeatedly apply triangle inequality
\begin{align*}   &\bigg(\sum_{i=1}^{N}\prod_{k=1}^{p} \left(f^{(k)}(X_i^{(k)})\right)^2\bigg)^{1/2}
\\&=\bigg(\sum_{i=1}^{N}\prod_{k=1}^{p} (f^{(k)})_i^2\bigg)^{1/2} \nonumber \\
    &\le \bigg(\sum_{i=1}^{N} (\pi_{s_1}f^{(1)})_i^2\prod_{k=2}^{p} (f^{(k)})_i^2\bigg)^{1/2}+\sum_{s\ge s_1}\bigg(\sum_{i=1}^{N} (\Delta_s f^{(1)})_i^2\prod_{k=2}^{p} (f^{(k)})_i^2\bigg)^{1/2}
  \nonumber \\ 
   &\le  \bigg(\sum_{i=1}^{N} (\pi_{s_1}f^{(1)})_i^2\prod_{k=2}^{p} (f^{(k)})_i^2\bigg)^{1/2}+\Bigg(\sum_{s\ge s_1}\bigg(\sum_{i=1}^{N} (\Delta_s f^{(1)})_i^2\bigg)^{1/2}\Bigg)\bigg(\prod_{k=2}^{p} \sup_{f^{(k)}\in \F^{(k)}} \|f^{(k)}\|_{\ell^N_{\infty}} \bigg)\nonumber\\
   &\le  \bigg(\sum_{i=1}^{N} (\pi_{s_1}f^{(1)})_i^2 (\pi_{s_1}f^{(2)})_i^2\prod_{k=3}^{p} (f^{(k)})_i^2\bigg)^{1/2}+\sum_{s\ge s_1} \bigg(\sum_{i=1}^{N}(\pi_{s_1}f^{(1)})_{i}^2 (\Delta_s f^{(2)})_i^2 \prod_{k=3}^{p} (f^{(k)})_{i}^{2} \bigg)^{1/2}\nonumber\\
   &\quad +\Bigg(\sum_{s\ge s_1}\bigg(\sum_{i=1}^{N} (\Delta_s f^{(1)})_i^2\bigg)^{1/2}\Bigg)\bigg(\prod_{k=2}^{p} \sup_{f^{(k)}\in \F^{(k)}} \|f^{(k)}\|_{\ell^N_{\infty}} \bigg)\nonumber\\
   &\le \bigg(\sum_{i=1}^{N} (\pi_{s_1}f^{(1)})_i^2 (\pi_{s_1}f^{(2)})_i^2\prod_{k=3}^{p} (f^{(k)})_i^2\bigg)^{1/2}+ \Bigg(\sum_{s\ge s_1}\bigg(\sum_{i=1}^{N} (\Delta_s f^{(2)})_i^2\bigg)^{1/2}\Bigg)\bigg(\prod_{k\ne 2} \sup_{f^{(k)}\in \F^{(k)}} \|f^{(k)}\|_{\ell^N_{\infty}} \bigg)\nonumber\\
   &\quad +\Bigg(\sum_{s\ge s_1}\bigg(\sum_{i=1}^{N} (\Delta_s f^{(1)})_i^2\bigg)^{1/2}\Bigg)\bigg(\prod_{k=2}^{p} \sup_{f^{(k)}\in \F^{(k)}} \|f^{(k)}\|_{\ell^N_{\infty}} \bigg) \nonumber\\
   &=  \bigg(\sum_{i=1}^{N} (\pi_{s_1}f^{(1)})_i^2 (\pi_{s_1}f^{(2)})_i^2\prod_{k=3}^{p} (f^{(k)})_i^2\bigg)^{1/2}+\sum_{k=1}^{2}\Bigg(\sum_{s\ge s_1}\bigg(\sum_{i=1}^{N} (\Delta_s f^{(k)})_i^2\bigg)^{1/2}\Bigg)\bigg(\prod_{k'\ne k} \sup_{f^{(k')}\in \F^{(k')}} \|f^{(k')}\|_{\ell^N_{\infty}} \bigg).\nonumber 
\end{align*}

Iterating this argument, we deduce that 
\begin{align}\label{eq:key_lemma_aux1}
&\bigg(\sum_{i=1}^{N}\prod_{k=1}^{p} \left(f^{(k)}(X_i^{(k)})\right)^2\bigg)^{1/2} \nonumber\\
&\le \underbrace{\bigg(\sum_{i=1}^{N}\prod_{k=1}^{p}(\pi_{s_1}f^{(k)})_i^2\bigg)^{1/2}}_{=:T_1}+\underbrace{\sum_{k=1}^{p}\Bigg(\sum_{s\ge s_1} \bigg(\sum_{i=1}^{N} (\Delta_s f^{(k)})^2_i\bigg)^{1/2}\Bigg)\bigg(\prod_{k'\ne k}\sup_{f^{(k')}\in \F^{(k')}} \|f^{(k')}\|_{\ell^N_{\infty}}\bigg)}_{=:T_2}.
\end{align}


Now we analyze the term $T_1$. Let $I_{s_1-1}$ be the set of the $j_{s_1-1}-1$ largest coordinates of $(|\prod_{k=1}^{p}(\pi_{s_1} f^{(k)})_{i}|)_{i=1}^{N}$. By triangle inequality,
\begin{align*}
   T_1&= \bigg(\sum_{i=1}^{N}\prod_{k=1}^{p}(\pi_{s_1}f^{(k)})_i^2\bigg)^{1/2} \\
   &\le \bigg(\sum_{i\in I_{s_1-1}^c} \prod_{k=1}^{p}(\pi_{s_1}f^{(k)})_i^2\bigg)^{1/2}+\bigg(\sum_{i\in I_{s_1-1}} \prod_{k=1}^{p}(\pi_{s_1}f^{(k)})_i^2\bigg)^{1/2}\\
    &\le \bigg(\sum_{i\in I_{s_1-1}^c} \prod_{k=1}^{p}(\pi_{s_1}f^{(k)})_i^2\bigg)^{1/2}+\Bigg(\sum_{i\in I_{s_1-1}} \bigg(\prod_{k=1}^{p}(\pi_{s_1}f^{(k)})_i-\prod_{k=1}^{p}(\pi_{s_1-1}f^{(k)})_i\bigg)^2\Bigg)^{1/2}\\
    &\quad +\bigg(\sum_{i\in I_{s_1-1}} \prod_{k=1}^{p}(\pi_{s_1-1}f^{(k)})_i^2\bigg)^{1/2}.
\end{align*}

Using the telescoping sum $a_1 \cdots a_{p}-b_1 \cdots b_{p}=\sum_{k=1}^{p} a_1 \cdots a_{k-1}\left(a_k-b_k\right) b_{k+1} \cdots b_{p}$ and triangle inequality, we have
\begin{align*}
&\Bigg(\sum_{i\in I_{s_1-1}} \bigg(\prod_{k=1}^{p}(\pi_{s_1}f^{(k)})_i-\prod_{k=1}^{p}(\pi_{s_1-1}f^{(k)})_i\bigg)^2\Bigg)^{1/2}\\
&\le \sum_{k=1}^{p} \bigg(\sum_{i\in I_{s_1-1}} (\pi_{s_1} f^{(1)})_i^2\cdots (\pi_{s_1} f^{(k-1)})_i^2(\Delta_{s_1-1} f^{(k)})_i^2 (\pi_{s_1-1} f^{(k+1)})_i^2\cdots (\pi_{s_1-1} f^{(p)})_i^2\bigg)^{1/2}\\
&\le \sum_{k=1}^{p} \bigg(\sum_{i\in I_{s_1-1}} (\Delta_{s_1-1} f^{(k)})_i^2 \bigg)^{1/2} \|\pi_{s_1}f^{(1)}\|_{\ell_{\infty}^N}\cdots \|\pi_{s_1}f^{(k-1)}\|_{\ell_{\infty}^N} \|\pi_{s_1-1}f^{(k+1)}\|_{\ell_{\infty}^N}\cdots \|\pi_{s_1-1}f^{(p)}\|_{\ell_{\infty}^N} \\
&\le \sum_{k=1}^{p} \bigg(\sum_{i\in I_{s_1-1}}(\Delta_{s_1-1} f^{(k)})_i^2\bigg)^{1/2}\bigg(\prod_{k'\ne k}\sup_{f^{(k')}\in \F^{(k')}} \|f^{(k')}\|_{\ell^N_{\infty}}\bigg)\\
&\le \sum_{k=1}^{p} \bigg(\sum_{i<j_{s_1-1}}\left(\left(\Delta_{s_1-1} f^{(k)}\right)_i^*\right)^2\bigg)^{1/2}\bigg(\prod_{k'\ne k}\sup_{f^{(k')}\in \F^{(k')}} \|f^{(k')}\|_{\ell^N_{\infty}}\bigg).
\end{align*}
Moreover, since 
\begin{align*}
   \bigg(\sum_{i\in I_{s_1-1}^c} \prod_{k=1}^{p}(\pi_{s_1}f^{(k)})_i^2\bigg)^{1/2} &= \Bigg(\sum_{i\ge j_{s_1-1}}\bigg(\bigg(\prod_{k=1}^{p}\pi_{s_1}f^{(k)}\bigg)^*_i\bigg)^2\Bigg)^{1/2}
   \end{align*}
   and
   \begin{align*}
\bigg(\sum_{i\in I_{s_1-1}} \prod_{k=1}^{p}(\pi_{s_1-1}f^{(k)})_i^2\bigg)^{1/2} &\le \max_{|J|=j_{s_1-1}-1}\bigg(\sum_{j\in J} \prod_{k=1}^{p}(\pi_{s_1-1}f^{(k)})_j^2\bigg)^{1/2},  
\end{align*}
it follows that
\begin{align*}
    T_1&= \bigg(\sum_{i=1}^{N}\prod_{k=1}^{p}(\pi_{s_1}f^{(k)})_i^2\bigg)^{1/2}\\
    &\overset{\text{$(\star)$}}{=}\max_{|I|=j_{s_1}-1}\bigg(\sum_{i\in I} \prod_{k=1}^{p}(\pi_{s_1}f^{(k)})_i^2\bigg)^{1/2}\\
    &\le \Bigg(\sum_{i\ge j_{s_1-1}}\bigg(\bigg(\prod_{k=1}^{p}\pi_{s_1}f^{(k)}\bigg)^*_i\bigg)^2\Bigg)^{1/2}+\sum_{k=1}^{p} \bigg(\sum_{i<j_{s_1-1}}\left(\left(\Delta_{s_1-1} f^{(k)}\right)_i^*\right)^2\bigg)^{1 / 2}\bigg(\prod_{k'\ne k}\sup_{f^{(k')}\in \F^{(k')}} \|f^{(k')}\|_{\ell^N_{\infty}}\bigg)\\
    &\quad +\max_{|I|=j_{s_1-1}-1}\bigg(\sum_{i\in I} \prod_{k=1}^{p}(\pi_{s_1-1}f^{(k)})_i^2\bigg)^{1/2},
\end{align*}
where  $(\star)$ follows by the definition of $s_1$, i.e. $j_{s_1}=N+1$.

Repeating this argument for $s_0\le s\le s_1-2$, it follows that for every $f^{(k)}\in \F^{(k)}$, $1\le k\le p$, and for the choice of $I_s$ as the set of $j_s-1$ largest coordinates of $(|\prod_{k=1}^{p}(\pi_{s+1} f^{(k)})_{i}|)_{i=1}^{N}$,
\begin{align*}
    &\max_{|I|=j_{s+1}-1}\bigg(\sum_{i\in I}\prod_{k=1}^{p}(\pi_{s+1} f^{(k)})^2_{i}\bigg)^{1/2}
    = \Bigg(\sum_{i=1}^{j_{s+1}-1}\bigg(\bigg(\prod_{k=1}^{p} \pi_{s+1} f^{(k)}\bigg)_i^*\bigg)^2\Bigg)^{1/2} \\
    &\le \Bigg(\sum_{i=j_{s}}^{j_{s+1}-1}\bigg(\bigg(\prod_{k=1}^{p}\pi_{s+1} f^{(k)}\bigg)_i^{*}\bigg)^{2}\Bigg)^{1 / 2} +\bigg(\sum_{i\in I_s}\prod_{k=1}^{p}(\pi_{s+1} f^{(k)})^2_{i}\bigg)^{1/2}\\
    &\le \Bigg(\sum_{i=j_{s}}^{j_{s+1}-1}\bigg(\bigg(\prod_{k=1}^{p}\pi_{s+1} f^{(k)}\bigg)_i^{*}\bigg)^{2}\Bigg)^{1 / 2}+\Bigg(\sum_{i\in I_{s}} \bigg(\prod_{k=1}^{p}(\pi_{s+1}f^{(k)})_i-\prod_{k=1}^{p}(\pi_{s}f^{(k)})_i\bigg)^2\Bigg)^{1/2}\\
    &\quad + \bigg(\sum_{i\in I_s}\prod_{k=1}^{p}(\pi_{s} f^{(k)})^2_{i}\bigg)^{1/2} \\
    &\le \Bigg(\sum_{i=j_{s}}^{j_{s+1}-1}\bigg(\bigg(\prod_{k=1}^{p}\pi_{s+1} f^{(k)}\bigg)_i^{*}\bigg)^{2}\Bigg)^{1 / 2}+ \sum_{k=1}^{p} \bigg(\sum_{i<j_{s}}\left(\left(\Delta_{s} f^{(k)}\right)_i^*\right)^2\bigg)^{1 / 2}\bigg(\prod_{k'\ne k}\sup_{f^{(k')}\in \F^{(k')}} \|f^{(k')}\|_{\ell^N_{\infty}}\bigg)\\
    &\quad + \max_{|I|=j_{s}-1}\bigg(\sum_{i\in I}\prod_{k=1}^{p}(\pi_{s} f^{(k)})^2_{i}\bigg)^{1/2}.
\end{align*}

Consequently, we express $T_1$ as a telescoping sum, from which we obtain that
\begin{align*}
T_1 &=\max_{|I|=j_{s_1}-1}\bigg(\sum_{i\in I} \prod_{k=1}^{p}(\pi_{s_1}f^{(k)})_i^2\bigg)^{1/2}\\
&=\sum_{s=s_0}^{s_1-1} \left( \max_{|I|=j_{s+1}-1}\bigg(\sum_{i\in I}\prod_{k=1}^{p}(\pi_{s+1} f^{(k)})^2_{i}\bigg)^{1/2}- \max_{|I|=j_{s}-1}\bigg(\sum_{i\in I}\prod_{k=1}^{p}(\pi_{s} f^{(k)})^2_{i}\bigg)^{1/2}\right)\\
&\quad +  \max_{|I|=j_{s_0}-1}\bigg(\sum_{i\in I}\prod_{k=1}^{p}(\pi_{s_0} f^{(k)})^2_{i}\bigg)^{1/2} \\
& \le \Bigg(\sum_{i\ge j_{s_1-1}}\bigg(\bigg(\prod_{k=1}^{p}\pi_{s_1}f^{(k)}\bigg)^*_i\bigg)^2\Bigg)^{1/2}+\sum_{s=s_0}^{s_1-2} \Bigg(\sum_{i=j_{s}}^{j_{s+1}-1}\bigg(\bigg(\prod_{k=1}^{p}\pi_{s+1} f^{(k)}\bigg)_i^{*}\bigg)^{2}\Bigg)^{1 / 2}\\
&\quad + \sum_{s=s_0}^{s_1-1}\sum_{k=1}^{p} \bigg(\sum_{i<j_{s}}\left(\left(\Delta_{s} f^{(k)}\right)_i^*\right)^2\bigg)^{1 / 2}\bigg(\prod_{k'\ne k}\sup_{f^{(k')}\in \F^{(k')}} \|f^{(k')}\|_{\ell^N_{\infty}}\bigg) \\
&\quad + \max_{|I|=j_{s_0}-1}\bigg(\sum_{i\in I}\prod_{k=1}^{p}(\pi_{s_0} f^{(k)})^2_{i}\bigg)^{1/2}.
\end{align*}
Notice that
\begin{align*}
\max_{|I|=j_{s_0}-1}\bigg(\sum_{i\in I}\prod_{k=1}^{p}(\pi_{s_0} f^{(k)})^2_{i}\bigg)^{1/2}&= \Bigg(\sum_{i<j_{s_0}}\bigg(\bigg(\prod_{k=1}^{p}\pi_{s_0} f^{(k)}\bigg)_i^*\bigg)^2\Bigg)^{1/2}\\
&\le \bigg(\sum_{i<j_{s_0}}\left(\left(\pi_{s_0} f^{(1)}\right)_i^*\right)^2\bigg)^{1/2}\bigg(\prod_{k'\ne 1}\sup_{f^{(k')}\in \F^{(k')}} \|f^{(k')}\|_{\ell^N_{\infty}}\bigg)\\
& \le \sum_{k=1}^{p} \bigg(\sum_{i<j_{s_0}}\left(\left(\pi_{s_0} f^{(k)}\right)_i^*\right)^2\bigg)^{1/2}\bigg(\prod_{k'\ne k}\sup_{f^{(k')}\in \F^{(k')}} \|f^{(k')}\|_{\ell^N_{\infty}}\bigg), 
\end{align*}
and so 
\begin{align}\label{eq:key_lemma_aux2}
T_1& \le \Bigg(\sum_{i\ge j_{s_1-1}}\bigg(\bigg(\prod_{k=1}^{p}\pi_{s_1}f^{(k)}\bigg)^*_i\bigg)^2\Bigg)^{1/2}+\sum_{s=s_0}^{s_1-2} \Bigg(\sum_{i=j_{s}}^{j_{s+1}-1}\bigg(\bigg(\prod_{k=1}^{p}\pi_{s+1} f^{(k)}\bigg)_i^{*}\bigg)^{2}\Bigg)^{1 / 2} \nonumber\\
&\quad + \sum_{k=1}^{p} \Bigg(\sum_{s=s_0}^{s_1-1} \bigg(\sum_{i<j_{s}}\left(\left(\Delta_{s} f^{(k)}\right)_i^*\right)^2\bigg)^{1 / 2}+\bigg(\sum_{i<j_{s_0}}\left(\left(\pi_{s_0} f^{(k)}\right)_i^*\right)^2\bigg)^{1/2}\Bigg)\bigg(\prod_{k'\ne k}\sup_{f^{(k')}\in \F^{(k')}} \|f^{(k')}\|_{\ell^N_{\infty}}\bigg).
\end{align}
Combining \eqref{eq:key_lemma_aux1} and \eqref{eq:key_lemma_aux2} gives that
\begin{align*}
   &\bigg(\sum_{i=1}^{N}\prod_{k=1}^{p} (f^{(k)})_i^2\bigg)^{1/2} \le T_1+T_2\\
    &\le \Bigg(\sum_{i\ge j_{s_1-1}}\bigg(\bigg(\prod_{k=1}^{p}\pi_{s_1}f^{(k)}\bigg)^*_i\bigg)^2\Bigg)^{1/2}+\sum_{s=s_0}^{s_1-2} \Bigg(\sum_{i=j_{s}}^{j_{s+1}-1}\bigg(\bigg(\prod_{k=1}^{p}\pi_{s+1} f^{(k)}\bigg)_i^{*}\bigg)^{2}\Bigg)^{1 / 2}\\
    &\quad + \sum_{k=1}^{p} \Bigg(\sum_{s=s_0}^{s_1-1} \bigg(\sum_{i<j_{s}}\left(\left(\Delta_{s} f^{(k)}\right)_i^*\right)^2\bigg)^{1 / 2}+\bigg(\sum_{i<j_{s_0}}\left(\left(\pi_{s_0} f^{(k)}\right)_i^*\right)^2\bigg)^{1/2}\Bigg)\bigg(\prod_{k'\ne k}\sup_{f^{(k')}\in \F^{(k')}} \|f^{(k')}\|_{\ell^N_{\infty}}\bigg)\\
    &\quad + \sum_{k=1}^{p}\Bigg(\sum_{s\ge s_1} \bigg(\sum_{i=1}^{N} (\Delta_s f^{(k)})^2_i\bigg)^{1/2}\Bigg)\bigg(\prod_{k'\ne k}\sup_{f^{(k')}\in \F^{(k')}} \|f^{(k')}\|_{\ell^N_{\infty}}\bigg)\\
    &=  \Bigg(\sum_{i\ge j_{s_1-1}}\bigg(\bigg(\prod_{k=1}^{p}\pi_{s_1}f^{(k)}\bigg)^*_i\bigg)^2\Bigg)^{1/2}+\sum_{s=s_0}^{s_1-2} \Bigg(\sum_{i=j_{s}}^{j_{s+1}-1}\bigg(\bigg(\prod_{k=1}^{p}\pi_{s+1} f^{(k)}\bigg)_i^{*}\bigg)^{2}\Bigg)^{1 / 2}\\
    &\quad +\sum_{k=1}^{p} \Bigg(\sum_{s\ge s_0} \bigg(\sum_{i<j_{s}}\left(\left(\Delta_{s} f^{(k)}\right)_i^*\right)^2\bigg)^{1 / 2}+\bigg(\sum_{i<j_{s_0}}\left(\left(\pi_{s_0} f^{(k)}\right)_i^*\right)^2\bigg)^{1/2}\Bigg)\bigg(\prod_{k'\ne k}\sup_{f^{(k')}\in \F^{(k')}} \|f^{(k')}\|_{\ell^N_{\infty}}\bigg).
\end{align*}





With the choice of $(j_{s})_{s\ge s_0}$ in \eqref{eq:j_s}, Corollary \ref{coro:1} gives that with probability at least $1-2\exp(-c_1u^2 2^{s_0})$, for every $1\le k\le p$, every $f^{(k)}\in \F^{(k)}$, and every $s\ge s_0$,
\begin{align*}
\bigg(\sum_{i<j_{s}}\left(\left(\Delta_{s} f^{(k)}\right)_i^*\right)^2\bigg)^{1 / 2}
&\lesssim  u 2^{s / 2} \|\Delta_s f^{(k)}\|_{(c_3 u^2 2^s)},\\ 
\bigg(\sum_{i<j_{s}}\left(\left(\pi_{s} f^{(k)}\right)_i^*\right)^2\bigg)^{1 / 2}&\lesssim  u 2^{s / 2}\|\pi_s f^{(k)}\|_{(c_3 u^2 2^s)},\\
\Bigg( \sum_{i\ge j_s}  \bigg(\bigg( \prod_{k=1}^{p}\pi_{s+1} f^{(k)} \bigg)^*_{i}\bigg)^2\Bigg)^{1/2} &  \lesssim N^{1/2} \bigg\|\prod_{k=1}^{p}\pi_{s+1} f^{(k)}\bigg\|_{L_4}.
\end{align*}
By Lemma \ref{lemma:infinity_norm_bound} and a union bound, with probability at least $1-p\exp(-u^2)$, for every $1\le k \le p$,
\[
\sup_{f^{(k)}\in \F^{(k)}} \|f^{(k)}\|_{\ell^N_{\infty}}\lesssim \gamma(\F^{(k)},\psi_2)+d_{\psi_2}(\F^{(k)}) (\sqrt{\log N}+u).
\]
Moreover, by Corollary \ref{coro:order_stats}, with probability at least $1-2\exp(-c_1 u^2 2^{s_0})$, for every $f^{(k)}\in \F^{(k)} \, (1\le k\le p)$ and every $s\ge s_0$,
\[
\Bigg(\sum_{i=j_{s-1}}^{j_{s}-1}\bigg(\bigg(\prod_{k=1}^{p}\pi_{s} f^{(k)}\bigg)_i^{*}\bigg)^{2}\Bigg)^{1 / 2}\lesssim \bigg\|\prod_{k=1}^{p}\pi_s f^{(k)}\bigg\|_{L_4}  \Bigg(\sum_{i=j_{s-1}}^{j_{s}-1}\left(\frac{e N}{i}\right)^{3/4}\Bigg)^{1 / 2}.
\]
Hence, it holds with probability $1-c_1(p)\exp(-c_2(p) u^2 2^{s_0})$ that, for every $f^{(k)}\in \F^{(k)} (1\le k\le p)$:
\begin{align}\label{eq:key_lemma_aux_special}
   &\bigg(\sum_{i=1}^{N}\prod_{k=1}^{p} (f^{(k)})_i^2\bigg)^{1/2}\nonumber\\
   &\le  \Bigg(\sum_{i\ge j_{s_1-1}}\bigg(\bigg(\prod_{k=1}^{p}\pi_{s_1}f^{(k)}\bigg)^*_i\bigg)^2\Bigg)^{1/2}+\sum_{s=s_0}^{s_1-2} \Bigg(\sum_{i=j_{s}}^{j_{s+1}-1}\bigg(\bigg(\prod_{k=1}^{p}\pi_{s+1} f^{(k)}\bigg)_i^{*}\bigg)^{2}\Bigg)^{1 / 2}\nonumber\\
    &\quad +\sum_{k=1}^{p} \Bigg(\sum_{s\ge s_0} \bigg(\sum_{i<j_{s}}\left(\left(\Delta_{s} f^{(k)}\right)_i^*\right)^2\bigg)^{1 / 2}+\bigg(\sum_{i<j_{s_0}}\left(\left(\pi_{s_0} f^{(k)}\right)_i^*\right)^2\bigg)^{1/2}\Bigg)\bigg(\prod_{k'\ne k}\sup_{f^{(k')}\in \F^{(k')}} \|f^{(k')}\|_{\ell^N_{\infty}}\bigg) \nonumber\\
    &\lesssim_p N^{1/2} \bigg\|\prod_{k=1}^{p}\pi_{s_1} f^{(k)}\bigg\|_{L_4}+\sum_{s=s_0}^{s_1-2}  \bigg\|\prod_{k=1}^{p}\pi_{s+1} f^{(k)}\bigg\|_{L_4}  \bigg(\sum_{i=j_{s}}^{j_{s+1}-1}\left(\frac{e N}{i}\right)^{3/4}\bigg)^{1 / 2}\nonumber\\
    &\quad + \sum_{k=1}^{p} u\bigg(\sum_{s\ge s_0}  2^{s/2} \|\Delta_s f^{(k)}\|_{(c_3 u^2 2^s)}+ 2^{s_0/2} \|\pi_{s_0}f^{(k)}\|_{(c_3 u^2 2^{s_0})}\bigg)\bigg(\prod_{k'\ne k}\sup_{f^{(k')}\in \F^{(k')}} \|f^{(k')}\|_{\ell^N_{\infty}}\bigg)\\
    &\lesssim_p N^{1/2} \bigg(\prod_{k=1}^{p} d_{\psi_2}(\F^{(k)})\bigg)+u\sum_{k=1}^{p} \widetilde{\Lambda}_{s_0, u}(\F^{(k)}) \bigg(\prod_{k'\ne k} \Big(\gamma(\F^{(k')},\psi_2)+d_{\psi_2}(\F^{(k')}) (\sqrt{\log N} +u \Big) \bigg)\nonumber,
\end{align}
where in the last step we have used the definition of the graded $\gamma$-type functional $\widetilde{\Lambda}_{s_0, u}(\F^{(k)}) $ and the inequalities
\[
\bigg\|\prod_{k=1}^{p}\pi_{s} f^{(k)}\bigg\|_{L_4}\le \prod_{k=1}^{p}  \|\pi_{s} f^{(k)}\|_{L_{4p}}\lesssim_p \prod_{k=1}^{p} d_{\psi_2}(\F^{(k)})
\]
and 
\begin{align}\label{eq:key_lemma_aux3}
\sum_{s=s_0}^{s_1-2}\Bigg(\sum_{i=j_{s}}^{j_{s+1}-1}\left(\frac{e N}{i}\right)^{3/4}\Bigg)^{1 / 2}\lesssim N^{1/2}.
\end{align}
The inequality \eqref{eq:key_lemma_aux3} follows from the facts that $j_{s_1}=N+1$ and $j_{s}>2j_{s-1}$ whenever $1<j_{s-1}, j_{s}<N+1$, as can be seen from \eqref{eq:j_s}. Indeed, we estimate:
\begin{align*}
\sum_{s=s_0}^{s_1-2}\bigg(\sum_{i=j_{s}}^{j_{s+1}-1} \left(\frac{eN}{i}\right)^{3/4}\bigg)^{1/2}&\le (eN)^{3/8} \sum_{s=s_0}^{s_1-2} 2\left((j_{s+1}-1)^{1/4}-(j_{s}-1)^{1/4}\right)^{1/2}
 \lesssim N^{3/8} \sum_{s=s_0}^{s_1-2}(j_{s+1}-1)^{1/8}\\
 &\lesssim N^{3/8} (j_{s_1-1}-1)^{1/8}\bigg(1+\sum_{s\ge 1}2^{-s/8}\bigg) \lesssim N^{1/2}.
 \end{align*}

Taking $s_0\asymp 1$, we have
\[
\widetilde{\Lambda}_{s_0, u}(\F^{(k)})= \Lambda_{s_0, u}(\F^{(k)}) + 2^{s_0/2}\sup_{f^{(k)} \in \F^{(k)}}\left\|\pi_{s_0} f^{(k)}\right\|_{(u^2 2^{s_0})}  \lesssim \gamma(\F^{(k)},\psi_2)+2^{s_0/2} d_{\psi_2}(\F^{(k)})\lesssim \gamma(\F^{(k)},\psi_2).
\]
Consequently, it holds with probability at least $1-c_1(p) \exp(-c_2(p)u^2)$ that
\begin{align*}
 &\sup_{f^{(k)}\in \F^{(k)},1\le k\le p} \bigg(\sum_{i=1}^{N}\prod_{k=1}^{p} \left(f^{(k)}(X_i^{(k)})\right)^2\bigg)^{1/2}=\sup_{f^{(k)}\in \F^{(k)},1\le k\le p} \bigg(\sum_{i=1}^{N}\prod_{k=1}^{p} (f^{(k)})_i^2\bigg)^{1/2}\\
 &\lesssim_p  \bigg(\prod_{k=1}^{p} d_{\psi_2}(\F^{(k)})\bigg)  \bigg(\sqrt{N}+u\sum_{k=1}^{p} \bar{\gamma}(\F^{(k)},\psi_2)\prod_{k'\ne k}\Big(\bar{\gamma}(\F^{(k')},\psi_2)+(\log N)^{1/2}+u\Big)\bigg),
\end{align*}
where $\bar{\gamma}(\F^{(k)},\psi_2):=\gamma(\F^{(k)},\psi_2)/d_{\psi_2}(\F^{(k)})$ for $1\le k\le p$. This completes the proof.
\end{proof}

\subsection{Log-concave ensembles}\label{subsec:log-concave}

 In this subsection, we briefly outline how the proofs of Theorem \ref{thm:main1} and Theorem \ref{thm:main2} can be modified to establish the result stated in Remark \ref{rem:log_concave} for isotropic, unconditional, log-concave ensembles.

For integer $p\ge 2$ and $1 \leq k \leq p$, let $X^{(k)}, X_1^{(k)}, \ldots, X_N^{(k)}$ be i.i.d. isotropic, unconditional, log-concave random vectors in $\R^{d_k}$. For any function $f$ defined on $(\Omega,\mu)$, we introduce the Orlicz $\psi_1$-norm of $f$ given by
\[
\|f\|_{\psi_1}:=\inf \left\{c>0: 
\mathbb{E}_{X \sim \mu}\insquare{ \exp \inparen{\frac{|f(X)|}{c}}} \leq 2\right\},
\]
with the convention $\inf \emptyset = \infty.$  We define $
d_{\psi_1}(\F) := \sup_{f \in \F}\|f\|_{\psi_1}$. 

Recall that in the proof of Theorem \ref{thm:main2}, we established in \eqref{eq:main2_aux_special} that it holds with probability at least $1-2\exp(-u^2)$ that
\begin{align*}
    \bigg|\sum_{i=1}^N \varepsilon_i \prod_{k=1}^{p} f^{(k)} (X^{(k)}_i)\bigg| 
    &\lesssim_p \sum_{k=1}^{p}  u \widetilde{\Lambda}_{s_0,cu}(\F^{(k)}) \sup_{f^{(k')}\in \F^{(k')},k'\ne k}\bigg(\sum_{i=1}^{N} \prod_{k'\ne k} (f^{(k')})^2_{i}\bigg)^{1/2} \\
    &\quad+uN^{1/2} \sum_{k=1}^{p} \bigg(\sum_{s\ge s_0} 2^{s/2}  \|\Delta_s f^{(k)}\|_{L_8} \bigg\|\prod_{k'<k}\pi_{s+1} f^{(k')}\prod_{k'>k}\pi_{s} f^{(k')}\bigg\|_{L_8}\bigg)\nonumber\\
    &\quad+ u2^{s_0/2} N^{1/2} \|\pi_{s_0}f^{(1)}\|_{L_8} \bigg\|\prod_{k'\ne 1}\pi_{s_0} f^{(k')}\bigg\|_{L_8}.
\end{align*}
Moreover, in the proof of Theorem \ref{thm:key_lemma}, we established in \eqref{eq:key_lemma_aux_special} that it holds with probability $1-c_1(p)\exp(-c_2(p) u^2 2^{s_0})$ that, for every $f^{(k)}\in \F^{(k)} \, (1\le k\le p)$,
\begin{align*}
&\bigg(\sum_{i=1}^{N}\prod_{k=1}^{p} \left(f^{(k)}(X_i^{(k)})\right)^2\bigg)^{1/2}\\
    &\lesssim_p N^{1/2} \bigg\|\prod_{k=1}^{p}\pi_{s_1} f^{(k)}\bigg\|_{L_4}+\sum_{s=s_0}^{s_1-2}  \bigg\|\prod_{k=1}^{p}\pi_{s+1} f^{(k)}\bigg\|_{L_4}  \bigg(\sum_{i=j_{s}}^{j_{s+1}-1}\left(\frac{e N}{i}\right)^{3/4}\bigg)^{1 / 2}\\
    &\quad + \sum_{k=1}^{p} u\bigg(\sum_{s\ge s_0}  2^{s/2} \|\Delta_s f^{(k)}\|_{(c_3 u^2 2^s)}+ 2^{s_0/2} \|\pi_{s_0}f^{(k)}\|_{(c_3 u^2 2^{s_0})}\bigg)\bigg(\prod_{k'\ne k}\sup_{f^{(k')}\in \F^{(k')}} \|f^{(k')}\|_{\ell^N_{\infty}}\bigg)\\
    &\lesssim_p N^{1/2}\bigg(\sup_{f^{(k)}\in \F^{(k)},1\le k\le p}\bigg\|\prod_{k=1}^{p}f^{(k)}\bigg\|_{L_4}\bigg)+u\sum_{k=1}^{p} \widetilde{\Lambda}_{s_0, u}(\F^{(k)}) \bigg(\prod_{k'\ne k}\sup_{f^{(k')}\in \F^{(k')}} \|f^{(k')}\|_{\ell^N_{\infty}}\bigg).
\end{align*}
Combining the two bounds and taking $s_0\asymp 1$ gives that, with probability at least $1-\exp(-u^2),$
\begin{align*}
     &\bigg|\sum_{i=1}^N \varepsilon_i \prod_{k=1}^{p} f^{(k)} (X^{(k)}_i)\bigg| \\
    &\lesssim_p \sum_{k=1}^{p}  u \widetilde{\Lambda}_{s_0,cu}(\F^{(k)})\left[ N^{1/2}\bigg(\sup_{f^{(k')}\in \F^{(k')}}\bigg\|\prod_{k'\ne k}f^{(k')}\bigg\|_{L_4}\bigg)+u\sum_{k'\ne k} \widetilde{\Lambda}_{s_0, u}(\F^{(k')}) \bigg(\prod_{k''\ne k,k'}\sup_{f^{(k'')}\in \F^{(k'')}} \|f^{(k'')}\|_{\ell^N_{\infty}}\bigg)\right]\\
     &\quad+uN^{1/2} \sum_{k=1}^{p} \bigg(\sum_{s\ge s_0} 2^{s/2}  \|\Delta_s f^{(k)}\|_{L_8} \bigg\|\prod_{k'<k}\pi_{s+1} f^{(k')}\prod_{k'>k}\pi_{s} f^{(k')}\bigg\|_{L_8}\bigg)\nonumber\\
    &\quad+ u2^{s_0/2} N^{1/2} \|\pi_{s_0}f^{(1)}\|_{L_8} \bigg\|\prod_{k'\ne 1}\pi_{s_0} f^{(k')}\bigg\|_{L_8}.
\end{align*}
Applying the inequalities
\begin{align*}
   \sup_{f^{(k')}\in \F^{(k')}}\bigg\|\prod_{k'\ne k}f^{(k')}\bigg\|_{L_4} &\le \prod_{k'\ne k} \sup_{f^{(k')}\in \F^{(k')}}\|f^{(k')}\|_{L_{4(p-1)}}\lesssim_p \prod_{k'\ne k} d_{\psi_1}(\F^{(k')}),
\\
\bigg\|\prod_{k'<k}\pi_{s+1} f^{(k')}\prod_{k'>k}\pi_{s} f^{(k')}\bigg\|_{L_8} &\le \prod_{k'<k}  \|\pi_{s+1} f^{(k')}\|_{L_{8(p-1)}}  \prod_{k'>k} \|\pi_{s}f^{(k')}\|_{L_{8(p-1)}} \lesssim_p \prod_{k'\ne k} d_{\psi_1}(\F^{(k')}), \\
\sum_{s\ge s_0} 2^{s/2}  \|\Delta_s f^{(k)}\|_{L_8} &\lesssim \sum_{s\ge s_0} 2^{s/2}  \|\Delta_s f^{(k)}\|_{\psi_1}\lesssim \gamma(\F^{(k)},\psi_1), \\
\|\pi_{s_0}f^{(1)}\|_{L_8} \bigg\|\prod_{k'\ne 1} \pi_{s_0}f^{(k')}\bigg\|_{L_8} &\le \|\pi_{s_0}f^{(1)}\|_{L_8} \prod_{k'\ne 1} \|\pi_{s_0}f^{(k')}\|_{L_{8(p-1)}}\lesssim_p \prod_{k=1}^{p} d_{\psi_1}(\F^{(k)}),
\end{align*}   
we obtain
\begin{align*}
     &\bigg|\sum_{i=1}^N \varepsilon_i \prod_{k=1}^{p} f^{(k)} (X^{(k)}_i)\bigg| \\
    &\lesssim_p \sum_{k=1}^{p}  u \widetilde{\Lambda}_{s_0,cu}(\F^{(k)})\left[ N^{1/2}\prod_{k'\ne k}d_{\psi_1}(\F^{(k')})+u\sum_{k'\ne k} \widetilde{\Lambda}_{s_0, u}(\F^{(k')}) \bigg(\prod_{k''\ne k,k'}\sup_{f^{(k'')}\in \F^{(k'')}} \|f^{(k'')}\|_{\ell^N_{\infty}}\bigg)\right]\\
     &\quad+uN^{1/2} \sum_{k=1}^{p} \gamma(\F^{(k)},\psi_1)\prod_{k'\ne k}d_{\psi_1}(\F^{(k')})+ u N^{1/2} \prod_{k=1}^{p} d_{\psi_1}(\F^{(k)}).
\end{align*}

We use Borell's characterization of log-concave distributions \cite[Lemma 2.3]{adamczak2010quantitative}: if $X$ is a centered random vector in $\R^d$ with a log-concave distribution, then for any $v\in \S^{d-1}$,
\[
\|\langle X,v \rangle\|_{\psi_1}\asymp \|\langle X,v \rangle\|_{L_2}.
\]
Therefore, for $\F^{(k)}=\{\langle \cdot,v \rangle: v\in \S^{d_k-1} \}$, we have
\begin{align*}\label{eq:aux1}
d_{\psi_1}(\F^{(k)})=\sup_{f^{(k)}\in\mathcal{F}^{(k)}}\|f^{(k)}\|_{\psi_1}\asymp \sup_{f^{(k)}\in\mathcal{F}^{(k)}}\|f^{(k)}\|_{L_2}=\sup_{\|v\|= 1}\big(\E\langle X^{(k)},v\rangle^2\big)^{1/2}=1.
\end{align*}

For each $1\le k\le p$, let $Y^{(k)}$ be a standard Gaussian random vector in $\R^{d_k}$. Then, for any $u, v \in \S^{d_k-1}$, the canonical metric associated with $Y^{(k)}$ satisfies
\[
d_{Y^{(k)}}(u, v)
= \big(\mathbb{E}(\langle Y^{(k)}, u\rangle-\langle Y^{(k)}, v\rangle)^2\big)^{1/2}=\|u-v\|_{\ell_2}
=\|\langle\cdot, u\rangle-\langle\cdot, v\rangle\|_{L_2(\mu^{(k)})},
\]
where $\mu^{(k)}$ denotes the law of $X^{(k)}$, and the last step follows from the fact that $X^{(k)}$ is isotropic. Then, we have
\[
\gamma(\mathcal{F}^{(k)}, \psi_1)\overset{\text{$(\star)$}}{\asymp}\gamma(\mathcal{F}^{(k)}, L_2)=\gamma\big(\S^{d_k-1}, d_{Y^{(k)}}\big) \overset{\text{$(\star\star)$}}{\asymp} \mathbb{E} \sup _{v \in \S^{d_k-1}}\langle Y^{(k)}, v\rangle = \E \|Y^{(k)}\| \asymp \big(\E \|Y^{(k)}\|^2\big)^{1/2}=d_k^{1/2},
\]
where $(\star)$ follows from the fact that for log-concave random variables, the $\psi_1$-norm of any linear functional is equivalent to its $L_2$-norm, and $(\star\star)$ follows from Talagrand's majorizing-measure theorem \cite[Theorem 2.10.1]{talagrand2022upper}.

For $1\le k\le p$, let $Z^{(k)}$ be a random vector in $\R^{d_k}$ whose coordinates are independent standard exponential random variables. For $s_0\asymp 1$ and $u\ge 1$, it follows from \cite[Section 4.3]{mendelson2016upper} that, for the isotropic, unconditional, log-concave random vector $X^{(k)}$, and the linear function class $\F^{(k)}=\{\langle \cdot,v \rangle: v\in \S^{d_k-1} \}$, we have
\[
\widetilde{\Lambda}_{s_0,u}(\F^{(k)})\lesssim u\,\E \|Z^{(k)}\|\le u\big(\E \|Z^{(k)}\|^2 \big)^{1/2}=ud_k^{1/2}.
\]

By \cite[(2.20)]{zhivotovskiy2024dimension}, for each $k$, we have with probability at least $1-N\exp(-t)$,
\[
\max_{1\le i\le N}\|X^{(k)}_i\|\lesssim d_k^{1/2}+t.
\]
Setting $t=\log N+u^2$, we have with probability at least $1-N\exp(-\log N-u^2)= 1-\exp(-u^2)$ that
\[
\sup_{f^{(k)}\in \F^{(k)}}\|f^{(k)}\|_{\ell_{\infty}^N}=\sup_{f^{(k)}\in \F^{(k)}}\max_{1\le i\le N}|f^{(k)}(X_i^{(k)})|=\max_{1\le i\le N}\sup_{v\in \S^{d_k-1}} |\langle v, X^{(k)}_i \rangle|=\max_{1\le i\le N}\|X^{(k)}_i\|\lesssim d_k^{1/2}+\log N+u^2.
\]
Combining everything together, for any $u\ge 1$, we have with probability  at least $1-\exp(-u^2)$ that
\begin{align*}
     &\sup_{f^{(k)}\in \F^{(k)},1\le k\le p}\bigg|\sum_{i=1}^N \varepsilon_i \prod_{k=1}^{p} f^{(k)} (X^{(k)}_i)\bigg| \\
      &\lesssim_p \sum_{k=1}^{p}  u \widetilde{\Lambda}_{s_0,cu}(\F^{(k)})\left[ N^{1/2}\prod_{k'\ne k}d_{\psi_1}(\F^{(k')})+u\sum_{k'\ne k} \widetilde{\Lambda}_{s_0, u}(\F^{(k')}) \bigg(\prod_{k''\ne k,k'}\sup_{f^{(k'')}\in \F^{(k'')}} \|f^{(k'')}\|_{\ell^N_{\infty}}\bigg)\right]\\
     &\quad+uN^{1/2} \sum_{k=1}^{p} \gamma(\F^{(k)},\psi_1)\prod_{k'\ne k}d_{\psi_1}(\F^{(k')})+ u N^{1/2} \prod_{k=1}^{p} d_{\psi_1}(\F^{(k)})\\
    &\lesssim_p \sum_{k=1}^{p}  u^2 d_k^{1/2}\bigg( N^{1/2}+u^2\sum_{k'\ne k} d_{k'}^{1/2} \prod_{k''\ne k,k'}\big(d_{k''}^{1/2}+\log N +u^2\big)\bigg)+uN^{1/2} \sum_{k=1}^{p}d_k^{1/2}\\
     &\lesssim_p  u^2 N^{1/2} \bigg(\sum_{k=1}^{p}d_k^{1/2}\bigg)+u^4\sum_{k\ne k'} d_{k}^{1/2} d_{k'}^{1/2}\prod_{k''\ne k,k'} \big(d_{k''}^{1/2}+\log N +u^2\big).
\end{align*}
Integrating the tail bound yields the following bound in expectation:
\begin{align*}
\E \sup_{f^{(k)}\in \F^{(k)},1\le k\le p}\bigg|\sum_{i=1}^N \varepsilon_i \prod_{k=1}^{p} f^{(k)} (X^{(k)}_i)\bigg|&\lesssim_p N^{1/2} \bigg(\sum_{k=1}^{p}d_k^{1/2}\bigg)+\sum_{k\ne k'} d_{k}^{1/2} d_{k'}^{1/2}\prod_{k''\ne k,k'} \big(d_{k''}^{1/2}+\log N \big)\\
&\lesssim_p N^{1/2} \bigg(\sum_{k=1}^{p}d_k^{1/2}\bigg)+\prod_{k=1}^{p} \big(d_{k}^{1/2}+\log N\big).
\end{align*}
Finally, by the symmetrization inequality for empirical processes \cite[Lemma 2.3.1]{van2023weak},
\begin{align*}
&\E \bigg\|\frac{1}{N}\sum_{i=1}^{N} X^{(1)}_i\otimes \cdots \otimes X^{(p)}_{i} -\E\, X^{(1)}\otimes \cdots\otimes X^{(p)} \bigg\| \\
&= \E \sup_{\|v_1\|=\cdots =\|v_p\|=1} \bigg|\frac{1}{N}\sum_{i=1}^N \prod_{k=1}^{p} \langle X^{(k)}_i, v_k\rangle- \E \prod_{k=1}^{p} \langle X^{(k)} , v_k \rangle \bigg|\\
&=\E \sup_{f^{(k)} \in \mathcal{F}^{(k)},1\le k\le p}\bigg|\frac{1}{N} \sum_{i=1}^N \prod_{k=1}^p f^{(k)} (X^{(k)}_i)-\mathbb{E} \prod_{k=1}^p f^{(k)} (X^{(k)})\bigg|\\
&\le 2\E \sup_{f^{(k)} \in \mathcal{F}^{(k)},1\le k\le p}\bigg|\frac{1}{N} \sum_{i=1}^N \varepsilon_i \prod_{k=1}^p f^{(k)} (X^{(k)}_i)\bigg|\\
&\lesssim_p \bigg(\frac{\sum_{k=1}^p d_k}{N}\bigg)^{1/2}+\frac{1}{N}\prod_{k=1}^{p}\big(d_k^{1/2}+\log N\big).
\end{align*}











\section*{Funding}
This work was funded by NSF CAREER award NSF DMS-2237628. Part of this research was performed while the authors were visiting the Institute for Mathematical and Statistical Innovation (IMSI), which is supported by the National Science Foundation (Grant No. DMS-1929348).

\section*{Data availability}

No new data were generated or analyzed in support of this research.

\bibliographystyle{plain}
\bibliography{references}

\renewcommand{\theHsection}{A\arabic{section}}

\begin{appendix}

\section*{Appendix}
 
\section{Proofs of auxiliary results}\label{app:A}

\begin{lemma}\label{lemma:sup_Gaussian_products}
Let $Z^{(k)}_i, 1\le i\le N, 1\le k\le s$ be i.i.d. standard Gaussian random variables in $\R$. Then,
\[
\E  \max_{1\le j\le N} \prod_{k=1}^{s}|Z^{(k)}_j| \asymp_s (\log N)^{s/2}.
\]
\end{lemma}

\begin{proof}

For the upper bound, we have
\[
\E  \max_{1\le j\le N} \prod_{k=1}^{s}|Z^{(k)}_j|\le \E \prod_{k=1}^{s}\max_{1\le j\le N} |Z^{(k)}_j|=\prod_{k=1}^{s}\E \max_{1\le j\le N}|Z^{(k)}_j|\overset{\text{$(\star)$}}{\le} \left(C\sqrt{\log N}\right)^{s}\lesssim_s (\log N)^{s/2},
\]
where the inequality $(\star)$ follows by \cite[Lemma 2.3.4]{talagrand2022upper}. Now we prove the lower bound. Let $Z, Z^{(1)},\ldots,Z^{(s)}$ be i.i.d. standard Gaussian random variables in $\R$. By \cite[Proposition 2.1.2]{vershynin2018high}, there exists positive constants $c_1\in (0,1),c_2\in (1,\infty)$ such that, for any $t\ge 2$,
\begin{align*}
\mathbb{P}\left[|Z^{(1)}\cdots Z^{(s)}|\ge t\right] &\ge  \mathbb{P}\left[ |Z^{(1)}|\ge t^{1/s},\ldots, |Z^{(s)}|\ge t^{1/s} \right]= \left(\mathbb{P}\left[|Z|>t^{1/s}\right]\right)^s\\
&\ge \left(c_1 \exp(-c_2 t^{2/s})\right)^s =c_1^s \exp\left(-sc_2 t^{2/s}\right). 
\end{align*}
It follows that
\begin{align*}
\mathbb{P}\left[ \max_{1\le j\le N} \prod_{k=1}^{s}|Z^{(k)}_j| \ge t \right] = 1- \left(1-\mathbb{P}\left[ |Z^{(1)}\cdots Z^{(s)}| \ge t\right]\right)^N\ge 1-\left(1-c_1^s \exp\left(-sc_2 t^{2/s} \right)\right)^N.
\end{align*}
Taking $t=c_3(\log N)^{s/2}$ for some sufficiently small constant $c_3$ depending only on $s$, we obtain
\[
\mathbb{P}\left[ \max_{1\le j\le N} \prod_{k=1}^{s}|Z^{(k)}_j| \ge c_3(\log N)^{s/2} \right]\ge 1-\left(1-c_1^s \exp\left(-sc_2 c_3^{2/s} \log N \right)\right)^N\ge \frac{1}{2},
\]
from where it follows that 
$
\E \max_{1\le j\le N} \prod_{k=1}^{s}|Z^{(k)}_j|\ge \frac{c_3}{2}(\log N)^{s/2}. \qedhere
$
\end{proof}

\begin{lemma}\label{lemma:tail_to_moment}
Let $X\ge 0$ be a random variable. Fix $u_0\ge 1$ and let $F:[u_0,\infty)\to (0,\infty)$ be increasing. Assume there exist constants
$C_0\ge 1$, $c_0>0$, and $m\ge 1$ such that, for all $u\ge u_0$,
\begin{equation}\label{eq:tail_assump}
\mathbb{P}\big[X>F(u)\big] \le C_0 e^{-c_0 u^2},
\end{equation}
and, for all $u\ge v\ge u_0$,
\begin{equation}\label{eq:growth_assump}
\frac{F(u)}{F(v)} \le \Big(\frac{u}{v}\Big)^m.
\end{equation}
Then, for every $q\ge 1$,
\begin{equation}\label{eq:moment_conclusion}
\big(\E \, X^q\big)^{1/q}\ \le\ C\,F\!\big(C'\sqrt{q}+u_0\big),
\end{equation}
where $C,C'>0$ depend only on $C_0,c_0,m$, and are independent of $q$ and of any parameters implicit in $F$.
\end{lemma}

\begin{proof}
We use the tail integral identity
\begin{equation*}
\E\, X^q = q\int_{0}^{\infty} t^{q-1}\,\mathbb P [X>t]\,dt.
\end{equation*}
Set $t_0:=F(u_0)$. Splitting at $t_0$ gives
\[
\mathbb E\, X^q
\le q\int_0^{t_0} t^{q-1}\,dt \;+\; q\int_{t_0}^{\infty} t^{q-1}\mathbb P[X>t]\,dt
= t_0^q + I,
\]
where $I:=q\int_{t_0}^{\infty} t^{q-1}\mathbb P[X>t]\,dt$.

To bound $I$, use the monotone change of variables $t=F(u)$ and the tail assumption \eqref{eq:tail_assump}:
\[
I
\le \int_{u_0}^{\infty}\mathbb P\big[X>F(u)\big]\,d\!\big(F(u)^q\big)
\le C_0\int_{u_0}^{\infty}e^{-c_0u^2}\,d\!\big(F(u)^q\big).
\]
Integrating by parts yields
\[
\int_{u_0}^{\infty}e^{-c_0u^2}\,d\!\big(F(u)^q\big)
= -e^{-c_0 u_0^2} F(u_0)^q+ 2c_0\int_{u_0}^{\infty}u\,e^{-c_0u^2}\,F(u)^q\,du\le 2c_0\int_{u_0}^{\infty}u\,e^{-c_0u^2}\,F(u)^q\,du.
\]
Therefore,
\begin{equation}\label{eq:moment_reduce}
\E\, X^q \le F(u_0)^q + 2c_0C_0\int_{u_0}^{\infty}u\,e^{-c_0u^2}\,F(u)^q\,du.
\end{equation}

Now choose a constant $\alpha\ge 2u_0$ (to be fixed depending only on $c_0$ and $m$) and set
\[
u_\ast:=\alpha\sqrt q.
\]
Split the integral in \eqref{eq:moment_reduce} at $u_\ast$.
For the lower part, monotonicity of $F$ gives
\[
\int_{u_0}^{u_\ast}u\,e^{-c_0u^2}F(u)^q\,du
\le F(u_\ast)^q\int_{0}^{\infty}u\,e^{-c_0u^2}\,du
\lesssim_{c_0} F(u_\ast)^q.
\]
For the upper part, use the growth condition \eqref{eq:growth_assump}:
for $u\ge u_\ast\ge u_0$,
\[
F(u)\le F(u_\ast)\Big(\frac{u}{u_\ast}\Big)^m,
\]
hence
\[
\int_{u_\ast}^{\infty}u\,e^{-c_0u^2}F(u)^q\,du
\le F(u_\ast)^q \int_{u_\ast}^{\infty}u\,e^{-c_0u^2}\Big(\frac{u}{u_\ast}\Big)^{mq}\,du.
\]
Change variables $u=u_\ast s$ to get
\[
\int_{u_\ast}^{\infty}u\,e^{-c_0u^2}\Big(\frac{u}{u_\ast}\Big)^{mq}\,du
= u_\ast^2\int_{1}^{\infty}s^{mq+1}e^{-c_0u_\ast^2 s^2}\,ds
\le u_\ast^2\int_{1}^{\infty}s^{mq+1}e^{-c_0\alpha^2 q\, s^2}\,ds.
\]
Since $s^2\ge s$ for $s\ge 1$, we have $e^{-c_0\alpha^2 q s^2}\le e^{-c_0\alpha^2 q s}$ and thus
\[
u_\ast^2\int_{1}^{\infty}s^{mq+1}e^{-c_0\alpha^2 q\, s^2}\,ds
\le u_\ast^2\int_{1}^{\infty}s^{mq+1}e^{-c_0\alpha^2 q\, s}\,ds.
\]
With the substitution $t=c_0\alpha^2 q\,s$,
\[
\int_{1}^{\infty}s^{mq+1}e^{-c_0\alpha^2 q s}\,ds
= (c_0\alpha^2 q)^{-(mq+2)}\int_{c_0\alpha^2 q}^{\infty} t^{mq+1}e^{-t}\,dt
\le (c_0\alpha^2 q)^{-(mq+2)}\Gamma(mq+2).
\]
Using the bound $\Gamma(r+1)\le r^r$ for $r\ge 1$ (e.g.\ from Stirling) gives
\[
(c_0\alpha^2 q)^{-(mq+2)}\Gamma(mq+2)
\lesssim (c_0\alpha^2 q)^{-(mq+2)}(mq+1)^{mq+1}
\lesssim \Big(\frac{m}{c_0\alpha^2}\Big)^{mq}\cdot \frac{1}{q},
\]
up to absolute constants. Choosing $\alpha$ large enough so that
$m/(c_0\alpha^2)\le 1/2$ ensures the right-hand side is $\lesssim 1/q$, hence
\[
\int_{u_\ast}^{\infty}u\,e^{-c_0u^2}\Big(\frac{u}{u_\ast}\Big)^{mq}\,du
\lesssim u_\ast^2\cdot \frac{1}{q}
\asymp \alpha^2.
\]
Therefore,
\[
\int_{u_\ast}^{\infty}u\,e^{-c_0u^2}F(u)^q\,du \lesssim_{\alpha,c_0,m} F(u_\ast)^q.
\]
Plugging the two parts back into \eqref{eq:moment_reduce}, and using that $F(u_0)\le F(u_\ast)$ and that $\alpha$ is chosen depending only on $c_0$ and $m,$ yields
\[
\mathbb E\, X^q \ \lesssim_{C_0,c_0,m}\ F(u_\ast)^q.
\]
Taking $q$-th roots gives \eqref{eq:moment_conclusion} (absorbing constants into $C,C'$).
\end{proof}

\begin{example}[Growth condition for a polynomial with nonnegative coefficients]\label{ex:growth_positive_poly}
Let
\[
F(u)=\sum_{j=0}^{d} c_j u^j,\qquad c_j\ge 0.
\]
Then, for any $u\ge v>0$,
\[
F(u)\le \Big(\frac{u}{v}\Big)^d F(v),
\qquad\text{i.e.}\qquad
\frac{F(u)}{F(v)}\le \Big(\frac{u}{v}\Big)^d.
\]
Indeed, for each $j\le d$ we have
\[
u^j=v^j\Big(\frac{u}{v}\Big)^j\le v^j\Big(\frac{u}{v}\Big)^d,
\]
and multiplying by $c_j\ge 0$ and summing over $j$ gives
\[
F(u)=\sum_{j=0}^d c_j u^j
\le \Big(\frac{u}{v}\Big)^d \sum_{j=0}^d c_j v^j
=\Big(\frac{u}{v}\Big)^d F(v).
\]
\end{example}

\section{Optimality of expectation bound \eqref{eq:aux_key_lemma}}\label{app:B}

We demonstrate the optimality of the expectation bound \eqref{eq:aux_key_lemma} in Remark \ref{remark:optimality} via a simple example. 

For $1\le k\le p$, let $H^{(k)}=\R^{d_k}$ and consider a symmetric compact set $S^{(k)}\subset \R^{d_k}$. Define the function class $\F^{(k)}=\{\langle \cdot,v_k \rangle: v_k\in S^{(k)} \}$, which consists of linear functions indexed by $S^{(k)}$. Let $X^{(k)}_1,\ldots,X^{(k)}_N$ be i.i.d. standard Gaussian random vectors in $\R^{d_k}$ for each $1\le k\le p$, and assume that the sequences $(X^{(1)}_i)_{i=1}^{N},\ldots,(X_i^{(p)})_{i=1}^{N}$ are independent. Let $v_k^{*}\in S^{(k)}$ be a point attaining the maximal Euclidean norm on $\R^{d_k}$, i.e., $ \|v_k^{*}\|_{\ell_2}=\sup_{v_k\in S^{(k)}}\|v_k\|_{\ell_2}$. Observe that
\begin{align}\label{eq:remark_aux1}
    &\E \sup_{f^{(k)}\in \F^{(k)},1\le k\le p} \bigg(\sum_{i=1}^{N}\prod_{k=1}^{p} \left(f^{(k)}(X_i^{(k)})\right)^2\bigg)^{1/2} =\E \sup_{v_k\in S^{(k)},1\le k\le p} \bigg(\sum_{i=1}^{N}\prod_{k=1}^{p}\langle X_i^{(k)},v_k \rangle^2 \bigg)^{1/2}\nonumber\\
    &\gtrsim \E \bigg(\sum_{i=1}^{N}\prod_{k=1}^{p}\langle X_i^{(k)},v^{*}_{k} \rangle^2 \bigg)^{1/2}\gtrsim_p \sqrt{N}\prod_{k=1}^{p}\|v^{*}_{k}\|_{\ell_2}\asymp_p \sqrt{N}\bigg(\prod_{k=1}^{p} d_{\psi_2}(\F^{(k)})\bigg), 
\end{align}
where the last step follows by $\|v_k^{*}\|_{\ell_2}=\sup_{v_k\in S^{(k)}}\|v_k\|_{\ell_2}\asymp \sup_{f^{(k)}\in \F^{(k)}}\|f^{(k)}\|_{\psi_2}= d_{\psi_2}(\F^{(k)})$ for every $1\le k\le p$. This confirms the necessity of the term $\sqrt{N}\big(\prod_{k=1}^{p} d_{\psi_2}(\F^{(k)})\big)$ in \eqref{eq:aux_key_lemma}.

On the other hand, we may assume without loss of generality that
\[
\bar{\gamma}(\F^{(1)},\psi_2)\ge \cdots \ge \bar{\gamma}(\F^{(t)},\psi_2)\ge (\log N)^{1/2} \ge \bar{\gamma}(\F^{(t+1)},\psi_2)\ge \cdots \ge \bar{\gamma}(\F^{(p)},\psi_2).
\]
Then,
\begin{align*}
    &\E \sup_{v_k\in S^{(k)},1\le k\le p} \bigg(\sum_{i=1}^{N}\prod_{k=1}^{p}\langle X_i^{(k)},v_k \rangle^2 \bigg)^{1/2}
    \ge \E \sup_{v_k\in S^{(k)},1\le k\le p} \max_{1\le i\le N} \prod_{k=1}^{p} \big|\langle X_i^{(k)},v_k \rangle\big|\\
    &=\E \max_{1\le i\le N} \sup_{v_k\in S^{(k)},1\le k\le p}  \prod_{k=1}^{p}\big|\langle X_i^{(k)},v_k \rangle\big|   
    \overset{\text{($\star$)}}{\ge} \E   \max_{1\le i\le N}\sup_{v_k\in S^{(k)},1\le k\le t} \prod_{k=1}^{t} \big|\langle X_i^{(k)},v_k \rangle\big| \prod_{k=t+1}^{p} \big| \langle X_i^{(k)},v^*_k\rangle\big|\\
&=\bigg(\prod_{k=t+1}^{p}\|v_k^*\|_{\ell_2}\bigg)\E\left[\max_{1\le i\le N} \sup_{v_k\in S^{(k)},1\le k\le t}  \prod_{k=1}^{t} \big|\langle X_i^{(k)},v_k \rangle\big| \prod_{k=t+1}^{p} \bigg| \frac{\langle X_i^{(k)},v^*_k\rangle}{\|v^*_k\|_{\ell_2}}\bigg|\right]\\
&\ge \bigg(\prod_{k=t+1}^{p}\|v_k^*\|_{\ell_2}\bigg)\E_{(X^{(k)}_i)_{i=1}^{N},k\ge t+1}  \left[\max_{1\le i\le N}\bigg(\E_{(X^{(k)}_i)_{i=1}^{N},k\le t} \sup_{v_k\in S^{(k)},1\le k\le t} \prod_{k=1}^{t}\big|\langle X_i^{(k)},v_k \rangle\big|\bigg) \prod_{k=t+1}^{p} \bigg| \frac{\langle X_i^{(k)},v^*_k\rangle}{\|v^*_k\|_{\ell_2}}\bigg|
\right]\\
&=\bigg(\prod_{k=t+1}^{p}\|v_k^*\|_{\ell_2}\bigg)\E_{(X^{(k)}_i)_{i=1}^{N},k\ge t+1}  \left[\max_{1\le i\le N}\bigg(\prod_{k=1}^{t}\E \sup_{v_k\in S^{(k)}} \big|\langle X_i^{(k)},v_k \rangle\big|\bigg) \prod_{k=t+1}^{p} \bigg| \frac{\langle X_i^{(k)},v^*_k\rangle}{\|v^*_k\|_{\ell_2}}\bigg|
\right]\\
& \overset{\text{($\star\star$)}}{\asymp}_p \bigg(\prod_{k=t+1}^{p}d_{\psi_2}(\F^{(k)})\bigg) \bigg(\prod_{k=1}^{t}\gamma(\F^{(k)},\psi_2) \bigg)\E_{(X^{(k)}_i)_{i=1}^{N},k\ge t+1}  \left[\max_{1\le i\le N} \prod_{k=t+1}^{p} \bigg| \frac{\langle X_i^{(k)},v^*_k\rangle}{\|v^*_k\|_{\ell_2}}\bigg|\right]\\
&\asymp_p \bigg(\prod_{k=t+1}^{p}d_{\psi_2}(\F^{(k)})\bigg) \bigg(\prod_{k=1}^{t}\gamma(\F^{(k)},\psi_2) \bigg)(\log N)^{(p-t)/2},
\end{align*}
where ($\star$) follows by taking $v_{k}=v^*_{k}$ for $k\ge t+1$, and ($\star\star$) follows using that $\|v_k^{*}\|_{\ell_2}\asymp d_{\psi_2}(\F^{(k)})$ and that, by  Talagrand's majorizing-measure theorem \cite[Theorem 2.10.1]{talagrand2022upper}, $\E\sup_{v_k \in S^{(k)}}\langle X^{(k)},v_k \rangle\asymp \gamma(\F^{(k)},\psi_2)$. The last step follows since, for every $k\ge t+1$ and $1\le i\le N$, the random variables
$\frac{\langle X_i^{(k)},v^*_k\rangle}{\|v^*_k\|_{\ell_2}}$ are i.i.d. standard normal random variables, and these Gaussians are independent across $k$ and $i$. By applying Lemma \ref{lemma:sup_Gaussian_products}, we obtain
\[
\E_{(X^{(k)}_i)_{i=1}^{N},k\ge t+1}  \left[\max_{1\le i\le N} \prod_{k=t+1}^{p} \bigg| \frac{\langle X_i^{(k)},v^*_k\rangle}{\|v^*_k\|_{\ell_2}}\bigg|\right]\asymp_p (\log N)^{(p-t)/2}.
\]
Therefore, we have shown that
\begin{align}\label{eq:remark_aux2}
\E \sup_{v_k\in S^{(k)},1\le k\le p} \bigg(\sum_{i=1}^{N}\prod_{k=1}^{p}\langle X_i^{(k)},v_k \rangle^2 \bigg)^{1/2} &\gtrsim_p \bigg(\prod_{k=t+1}^{p}d_{\psi_2}(\F^{(k)})\bigg) \bigg(\prod_{k=1}^{t}\gamma(\F^{(k)},\psi_2) \bigg)(\log N)^{(p-t)/2}\nonumber\\
&\asymp_p \bigg(\prod_{k=1}^{p} d_{\psi_2}(\F^{(k)})\bigg)  \bigg(\prod_{k=1}^{p} \Big(\bar{\gamma}(\F^{(k)},\psi_2)+(\log N)^{1/2}\Big)\bigg).
\end{align}
Combining the two lower bounds \eqref{eq:remark_aux1} and \eqref{eq:remark_aux2} shows that, in this example, the upper bound \eqref{eq:aux_key_lemma} is sharp up to a constant depending only on $p$.

\end{appendix}

\end{document}